\documentclass[11pt]{article}
\usepackage{etex}
\reserveinserts{28}
\usepackage[all]{xy}

\usepackage{amsfonts}
\usepackage{amsthm}
\usepackage{enumerate}
\usepackage{graphicx}
\usepackage{mathrsfs}
\usepackage{bm}
\usepackage{cite}
\usepackage{amssymb,amsmath} 
\usepackage{makecell}
\usepackage{tikz}
\usepackage{pgf}
\usepackage{tikz}
\usetikzlibrary{patterns}
\usepackage{pgffor}
\usepackage{pgfcalendar}
\usepackage{pgfpages}
\usepackage{shuffle,yfonts}
\usepackage{mathtools}
\DeclareFontFamily{U}{shuffle}{}
\DeclareFontShape{U}{shuffle}{m}{n}{ <-8>shuffle7 <8->shuffle10}{}

\newcommand{\si}{\sigma}

\newcommand\Res{{\rm Res}}

\newcommand{\bfk}{{\boldsymbol{\sl{k}}}}

\newcommand{\bfx}{{\boldsymbol{\sl{x}}}}

\newcommand\bfsi{{\boldsymbol \sigma}}

 \allowdisplaybreaks
\usetikzlibrary{arrows,shapes,chains}
 \allowdisplaybreaks

\catcode`!=11
\let\!int\int \def\int{\displaystyle\!int}
\let\!lim\lim \def\lim{\displaystyle\!lim}
\let\!sum\sum \def\sum{\displaystyle\!sum}
\let\!sup\sup \def\sup{\displaystyle\!sup}
\let\!inf\inf \def\inf{\displaystyle\!inf}
\let\!cap\cap \def\cap{\displaystyle\!cap}
\let\!max\max \def\max{\displaystyle\!max}
\let\!min\min \def\min{\displaystyle\!min}
\let\!frac\frac \def\frac{\displaystyle\!frac}
\catcode`!=12

\let\oldsection\section
\renewcommand\section{\setcounter{equation}{0}\oldsection}

\allowdisplaybreaks

\DeclareMathOperator{\Li}{Li}

\DeclareMathOperator{\ti}{ti}

\DeclareMathOperator{\Ti}{Ti}

\def\N{\mathbb{N}}
\def\Z{\mathbb{Z}}
\def\Q{\mathbb{Q}}

\def\S{\widetilde{S}}

\theoremstyle{plain}
\newtheorem{thm}{Theorem}[section]
\newtheorem{lem}[thm]{Lemma}
\newtheorem{cor}[thm]{Corollary}
\newtheorem{con}[thm]{Conjecture}
\newtheorem{qu}{Question}[section]

\theoremstyle{definition}

\newtheorem{re}[thm]{Remark}
\newtheorem{exa}[thm]{Example}

\begin{document}
\title{\bf Contour Integrations and Parity Results of Cyclotomic Euler $T$-Sums and Multiple $t$-Values}
\author{
{Zhenlu Wang$^{a,}$\thanks{Email: zhenluwang@tzc.edu.cn}\quad{and}\quad Ce Xu$^{b,}$\thanks{Email: cexu2020@ahnu.edu.cn}}\\[1mm]
a. \small School of Artificial Intelligence, Taizhou University, \\ \small Taizhou 318000, {Zhejiang,} P.R. China\\
b. \small School of Mathematics and Statistics, Anhui Normal University,\\ \small Wuhu 241002, {Anhui,} P.R. China
\\[5mm]
\large \emph{Dedicated to Professor Masanobu Kaneko on the occasion of his 65th birthday}
}

\date{}
\maketitle

\noindent{\bf Abstract.} We will employ the method of contour integration to investigate the parity results of non-embedded cyclotomic multiple $t$-values, which we refer to as cyclotomic Euler $T$-sums. We can provide explicit parity formulas for the linear and quadratic cases of cyclotomic Euler $T$-sums, as well as state a parity theorem for the general case. We also present illustrative examples and corollaries. From this, some parity results for classical cyclotomic multiple $t$-values can be derived. Furthermore, we present several general formulas for cyclotomic Euler $T$-sums with denominators involving arbitrary rational polynomials through residue computations. By evaluating these polynomials and computing residues, many other formulas analogous to cyclotomic Euler $T$-sums can be derived. In particular, we also obtain certain parity results for the cyclotomic versions of multiple $T$-values as defined by Kaneko and Tsumura. Finally, we propose some conjectures and questions regarding the parity of cyclotomic multiple $t$-values and cyclotomic multiple $T$-values.
\medskip

\medskip


\section{Introduction}

In 1998, Flajolet and Salvy \cite{Flajolet-Salvy} systematically investigated the parity of Dirichlet-type series with numerators being products of harmonic numbers using the method of contour integration. These series are now referred to as \emph{classical Euler sums} and are defined in the following form:
\begin{align*}
&{S_{{p_1p_2\cdots p_k},q}} := \sum\limits_{n = 1}^\infty  {\frac{{H_n^{\left( {{p_1}} \right)}H_n^{\left( {{p_2}} \right)} \cdots H_n^{\left( {{p_k}} \right)}}}
{{{n^{q}}}}},
\end{align*}
where $p_j\in \N\ (j=1,2,\ldots,k)$ and $q\geq 2$. When $k=1$ and let $p_1=p$, $S_{p,q}$ is called a \emph{linear Euler sum} (now also known as a \emph{double zeta-star value}), and when $k>1$, it is referred to as a \emph{nonlinear Euler sum}. The quantity $p_1+\cdots+p_k+q$ is called the ``weight" of the sum, and the quantity $k$ is called the
``order". $H_n^{(p)}$ stands the {\emph{generalized harmonic number}} of order $p$ defined by
\[H_n^{(p)}:=\sum_{k=1}^n \frac{1}{k^p}\quad\text{and}\quad H_n^{(1)}\equiv H_n.\]
The well-known \emph{parity theorem} of Euler sums they proved states that: \emph{A Euler sum $S_{p_1\cdots p_k,q}$ with $k\geq 2$ reduces to a combination of sums of lower orders whenever the weight $p_1+\cdots+p_k+q$ and the order $k$ are of the same parity}. The primary method they employed to prove this theorem involved constructing contour integrals incorporating trigonometric functions, the digamma function, and rational functions, followed by evaluating all residue contributions to complete the proof.
As remarked by Flajolet and Salvy, every Euler sum of weight $w$ and degree $k$
is a $\Q$-linear combination of multiple zeta values (MZVs) of weight $w$ and depth at most $k+1$. For explicit formula, the readers may consult the Xu-Wang's paper \cite{Xu-Wang}. The  \emph{multiple zeta values} (MZVs) are defined by
\begin{align*}
\zeta(\bfk)\equiv \zeta(k_1,\ldots,k_r):=\sum_{0<n_1<\cdots<n_r} \frac{1}{n_1^{k_1}\cdots n_r^{k_r}}\in \mathbb{R},
\end{align*}
where $k_1,\ldots,k_r$ are positive integers and $k_r\geq 2$ (i.e. \emph{admissible}). Here $r$ and $k_1+\cdots+k_r$ are called the \emph{depth} and \emph{weight}, respectively. The systematic study of MZVs began in the early 1990s with the works of Hoffman \cite{H1992} and Zagier \cite{DZ1994}. Due to their surprising and sometimes mysterious appearance in the study of many branches of mathematics and theoretical physics, these special values have attracted a lot of attention and interest in the past three decades (for example, see Zhao's monograph \cite{Z2016}, which documents nearly all important research results on multiple zeta values discovered prior to 2016).

In \cite{Xu-Wang2022} and \cite{Xu-Wang2023}, Xu and Wang extended Flajolet and Salvy's contour integral method to investigate the following two classes of sums, known as (alternating) Euler $T$-sums and (alternating) Euler $\tilde{S}$-sums, respectively:
\begin{align}
&T_{p_1,p_2,\ldots,p_k,q}^{\si_1,\si_2,\ldots,\si_k,\si}
    =\sum_{n=1}^\infty\si^{n-1}\frac{h_{n-1}^{(p_1)}(\si_1)h_{n-1}^{(p_2)}(\si_2)\cdots
        h_{n-1}^{(p_k)}(\si_k)}{(n-1/2)^q}\,,\label{Tsum.Unify}\\
&\S_{p_1,p_2,\ldots,p_k,q}^{\si_1,\si_2,\ldots,\si_k,\si}
    =\sum_{n=1}^\infty\si^{n-1}\frac{h_n^{(p_1)}(\si_1)h_n^{(p_2)}(\si_2)\cdots
        h_n^{(p_k)}(\si_k)}{n^q}\,,\label{Stsum.Unify}
\end{align}
where $(p_1,p_2,\ldots,p_k,q)\in\mathbb{N}^{k+1}$ and $(\si_1,\si_2,\ldots,\si_k,\si)\in\{\pm 1\}^{k+1}$ with $(q,\si)\neq (1,1)$. The $h_n^{(p)}(\si)$ denotes the (alternating) odd harmonic number defined by
\begin{align*}
h_n^{(p)}(\si):=\sum_{k=1}^n \frac{\si^k}{(k-1/2)^p}.
\end{align*}
They established parity results for (alternating) Euler $T$-sums and (alternating) Euler $\S$-sums (see \cite[Thms. 54 and 55]{Xu-Wang2023}) by defining a new digamma function and constructing associated contour integrals. In particular, they derived explicit formulas for double and triple (alternating) Euler $T$-sums and $\S$-sums. By exploring their relationships with Hoffman's multiple $t$-values and Kaneko-Tsumura's multiple $T$-values, they further obtained parity formulas for double and triple (alternating) $t$-values and $T$-values (see \cite[Thms. 40 and 52]{Xu-Wang2023}). For $\bfk:=(k_1,\ldots,k_r)\in \N^r$ and $\bfsi:=(\si_1,\ldots,\si_r)\in \{\pm 1\}^r$ and $(k_r,\si_r)\neq (1,1)$, the (alternating) \emph{multiple $t$-values} (MtVs) are defined by (\cite{H2019,Xu-Wang2023})
\begin{align}
t(\bfk;\bfsi):=\sum_{0<n_1<\cdots<n_r} \frac{\si_1^{n_1}\cdots \si_r^{n_r}}{(n_1-1/2)^{k_1}\cdots (n_r-1/2)^{k_r}}.
\end{align}
For $\bfk:=(k_1,\ldots,k_r)\in \N^r$ and $\bfsi:=(\si_1,\ldots,\si_r)\in \{\pm 1\}^r$ and $(k_r,\si_r)\neq (1,1)$, the (alternating) \emph{multiple $T$-values} (MTVs) are defined by (\cite{KanekoTs2019,Xu-Wang2023})
\begin{align}
T(\bfk;\bfsi):=2^r\sum\limits_{0<n_1<\cdots<n_r} \frac{\si_1^{n_1}\cdots \si_r^{n_r}}{(2n_1-1)^{k_1} (2n_2-2)^{k_2}\cdots(2n_r-r)^{k_r}}.
\end{align}
In particular, $t(\bfk)=t(\bfk;\{1\}_r)$ and $T(\bfk)=T(\bfk;\{1\}_r)$ are the classical multiple $t$-values and multiple $T$-values, respectively. Here $\{1\}_r$ denotes the sequence obtained by repeating $1$ exactly $r$ times.
Recent research achievements on multiple $t$-values and multiple $T$-values have been remarkably prolific, with applications even extending to motive theory. For some recent related work, see references \cite{Charlton2021,CharltonHoffman-MathZ2025,KomatsuLuca2025,LiSong2023,LiYan2023,LiYan2025,Murakami2021,Zhao2015,Z2024} and other relevant literature.

Very recently, Rui and Xu \cite{Rui-Xu2025} established parity results for cyclotomic Euler sums by constructing extended trigonometric functions and digamma functions. By considering contour integrals involving these functions, they further derived some explicit formulas related to multiple polylogarithms. The \emph{cyclotomic Euler sum} is defined by
\begin{align}
S_{p_1,\ldots, p_k;q}(x_1,\ldots,x_k;x):=\sum_{n=1}^\infty \frac{\zeta_{n}(p_1;x_1)\zeta_{n}(p_2;x_2)\cdots \zeta_{n}(p_k;x_k)}{n^q}x^n,
\end{align}
where $p_1,\ldots,p_k,q\in \N$ and $x_1,\ldots,x_k,x$ are all roots of unity with $(q,x)\neq (1,1)$. In particular, if $k=0$, we denote $S_{\emptyset;q}(\emptyset;x):=\Li_q(x)$. Here $\zeta_n(p;x)$ stands the finite sum of polylogarithm function defined by
\begin{align}
\zeta_n(p;x):=\sum_{k=1}^n \frac{x^k}{k^p}\quad (p\in \N,\ |x|\leq 1),
\end{align}
and the \emph{polylogarithm function} $\Li_p(x)$ is defined by
\begin{align}
\Li_{p}(x):=\lim_{n\rightarrow \infty}\zeta_n(p;x)= \sum_{n=1}^\infty \frac{x^n}{n^p}\quad (|x|\leq 1,\ (p,x)\neq (1,1),\ p\in \N).
\end{align}
For any $(k_1,\dotsc,k_r)\in\N^r$, the classical \emph{multiple polylogarithm function} with $r$-variables is defined by
\begin{align}\label{defn-mpolyf}
\Li_{k_1,\dotsc,k_r}(x_1,\dotsc,x_r):=\sum_{0<n_1<\cdots<n_r} \frac{x_1^{n_1}\dotsm x_r^{n_r}}{n_1^{k_1}\dotsm n_r^{k_r}}
\end{align}
which converges if $|x_j\cdots x_r|<1$ for all $j=1,\dotsc,r$. It can be analytically continued to a multi-valued meromorphic function on $\mathbb{C}^r$ (see \cite{Zhao2007d}). In particular, if $(k_1,\dotsc, k_r)\in\N^r$ and $x_1,\ldots,x_r$ are $N$th roots of unity, we call them \emph{cyclotomic multiple zeta values of level $N$} which converges if $(k_r,x_r)\ne (1,1)$ (see \cite{YuanZh2014a} and \cite[Ch. 15]{Z2016}). Zhao \cite{Zhao2010} proposed a basis conjecture for level 3 and level 4 cyclotomic multiple zeta values. Li \cite{Li2024} partially proved this conjecture using motivic theory. And the level 2 cyclotomic multiple zeta values are called \emph{alternating multiple zeta values} (AMZVs) (see \cite{BorweinBrBr1997,Zhao2016} etc), generally denoted by the symbol $\zeta(\bfk;\bfx):= \Li_\bfk(\bfx)$ for $(x_1,\dotsc,x_r)\in\{\pm 1\}^r$ and $(k_r,x_r)\ne (1,1)$.

In this paper, we define the non-embedded cyclotomic multiple $t$-values--\emph{cyclotomic Euler $T$-sums} of the following form:
\begin{align}
T_{p_1,\ldots, p_k;q}(x_1,\ldots,x_k;x):=\sum_{n=1}^\infty \frac{t_{n}(p_1;x_1)t_{n}(p_2;x_2)\cdots t_{n}(p_k;x_k)}{(n-1/2)^q}x^n,
\end{align}
where $p_1,\ldots,p_k,q\in \N$ and $x_1,\ldots,x_k,x$ are all roots of unity with $(q,x)\neq (1,1)$. Similar to classical Euler sums, we refer to the quantity $p_1+\cdots+p_k+q$ as the ``weight" of the sum, and the quantity $k$ as the ``order". If $k=0$, we denote $T_{\emptyset;q}(\emptyset;x):=\ti_q(x)$.
Here $t_n(p;x)$ denotes the finite sum of $t$-polylogarithm function defined by
\begin{align}
t_n(p;x):=\sum_{k=1}^n \frac{x^k}{(k-1/2)^p},
\end{align}
and the $t$-polylogarithm function $\ti_p(x)$ is defined by
\begin{align}
\ti_{p}(x):=\lim_{n\rightarrow \infty}t_n(p;x)= \sum_{n=1}^\infty \frac{x^n}{(n-1/2)^p}\quad (|x|\leq 1,\ (p,x)\neq (1,1),\ p\in \N).
\end{align}
In this paper, we employ the methods developed by Rui and Xu to investigate parity results for cyclotomic Euler $T$-sums, and consequently establish certain parity conclusions regarding cyclotomic multiple $t$-values. For $\bfk=(k_1,\ldots,k_r)\in \N^r$ and $\bfx=(x_1,\ldots,x_r)$ (all $x_j$ are $N$-th roots of unity) with $(k_r,x_r)\neq (1,1)$, the \emph{cyclotomic multiple $t$-value of level $N$} $\ti_{k_1,\ldots,k_r}(x_1,\ldots,x_r)$ is defined by
\begin{align}
\ti_{\bfk}(\bfx):=\sum_{0<n_1<\cdots<n_r} \frac{x_1^{n_1}\cdots x_r^{n_r}}{(n_1-1/2)^{k_1}\cdots (n_r-1/2)^{k_r}}.
\end{align}
They are called \emph{cyclotomic multiple $t$-values} if $x_1,\ldots,x_r$ are any set of roots of unity.
It is evident that the above multiple series also converges for $|x_j\cdots x_r|<1\ (j=1,2,\ldots,r)$, in which case we call the series a \emph{multiple $t$-polylogarithm function}.
Obviously, by applying the stuffle relations (see \cite{H2000}), it can be shown that cyclotomic Euler $T$-sums can be expressed as $\Z$-coefficient linear combinations of cyclotomic multiple $t$-values. As an example, we have
\begin{align*}
T_{p_1,p_2;q}(x_1,x_2;x)&=\ti_{p_1,p_2,q}(x_1,x_2,x)+\ti_{p_2,p_1,q}(x_2,x_1,x)+\ti_{p_1+p_2,q}(x_1x_2,x)\\
&\quad+\ti_{p_1,p_2+q}(x_1,x_2x)+\ti_{p_2,p_1+q}(x_2,x_1x)+\ti_{p_1+p_2+q}(x_1x_2x).
\end{align*}
The main result of this paper is to establish the following parity theorem for cyclotomic Euler $T$-sums (see Theorem \ref{thm-parityc-CES-one}).
\begin{thm}\label{thm-parityc-theorem} Let $r\in\N$ and $x,x_1,\ldots,x_r$ be roots of unity, and $p_1,\ldots,p_r,q\geq 1$ with $(p_j,x_j)$ and $ (q,x)\neq (1,1)$. Then
\begin{align*}
&T_{p_1,p_2,\ldots,p_r;q}\Big(x_1,x_2,\ldots,x_r;x\Big)\\
&= (-1)^{p_1+p_2+\cdots+p_r+q+r-1}(xx_1\cdots x_r)T_{p_1,p_2,\ldots,p_r;q}\Big(x_1^{-1},x_2^{-1},\ldots,x_r^{-1};x^{-1}\Big) \quad(\text{mod  products}),
\end{align*}
where the ``mod products" means discarding all product terms of cyclotomic Euler sums and cyclotomic Euler $T$-sums with order less than $r$.
\end{thm}
Further, based on our calculations and observations, we propose the following parity conjecture concerning cyclotomic multiple $t$-values:
\begin{con}\label{conjparitycmtv} Let $r>1$ and $x_1,\ldots,x_r$ be roots of unity, and $k_1,\ldots,k_r\geq 1$ with $(k_r,x_r)\neq (1,1)$. If $x_1,\ldots,x_r\in\{z\in \mathbb{C}: z^N=1\}$, then
\begin{align*}
&\ti_{k_1,\ldots,k_r}(x_1,\ldots,x_r)=(-1)^{k_1+\cdots+k_r+r}(x_1\cdots x_r)\ti_{k_1,\ldots,k_r}\Big(x_1^{-1},\ldots,x_r^{-1}\Big) \quad(\text{mod  products}),
\end{align*}
where the ``mod products" means discarding all product terms of cyclotomic multiple zeta values and cyclotomic multiple $t$-values (which must appear) with depth less than $r$ and level less than or equal to $N$.
\end{con}

In particular, Corollaries \ref{cordoublecmtv} and \ref{cor-quadratic-C-ES-one} partially confirm the correctness of this Conjecture \ref{conjparitycmtv}.
It should be emphasized that Panzer \cite[Thm 1.3]{Panzer2017} proved the following parity properties of multiple polylogarithms:  for all $r\in \N$ and $\bfk=(k_1,\ldots,k_r)\in \N^r$, the function
\begin{align*}
\Li_{\bfk}(z_1,z_2,\ldots,z_r)-(-1)^{k_1+\cdots+k_r+r}\Li_{\bfk}(1/z_1,1/z_2,\ldots,1/z_r)
\end{align*}
is of depth at most $r-1$. Here $(z_1,\ldots,z_r)\in \mathbb{C}^r \setminus \bigcup_{1\leq i\leq j\leq r}\{(z_1,\ldots,z_r): z_iz_{i+1}\cdots z_j\in[0,+\infty)\}$. It is also worth noting that Panzer's paper does not provide a general formula for the parity of multiple polylogarithms, while recently Hirose \cite{Hirose2025} and Umezawa \cite{Umezawa2025arxiv} have respectively given explicit formulas for the parity of multiple zeta values and the parity of multiple polylogarithms.
Utilizing Panzer's parity result, a weakened form of Conjecture \ref{conjparitycmtv} can be presented:
\begin{thm}\label{thm-weakendedofCMtVparity}
Let $r>1$ and $x_1,\ldots,x_r$ be roots of unity, and $k_1,\ldots,k_r\geq 1$ with $(k_r,x_r)\neq (1,1)$. If $x_1,\ldots,x_r\in\{z\in \mathbb{C}: z^N=1\}$, then
\begin{align*}
&\ti_{k_1,\ldots,k_r}(x_1,\ldots,x_r)-(-1)^{k_1+\cdots+k_r+r}(x_1\cdots x_r)\ti_{k_1,\ldots,k_r}\Big(x_1^{-1},\ldots,x_r^{-1}\Big)
\end{align*}
can be expressed in terms of a $\Q$-linear combination of cyclotomic multiple zeta values with depth less than $r$ and level less than or equal to $2N$.
\end{thm}
\begin{proof}
According to the definition, it is not difficult to observe that cyclotomic multiple $t$-values can be expressed as $\Z$-coefficient linear combinations of cyclotomic multiple zeta values as follows:
\begin{align*}
\ti_{k_1,\ldots,k_r}(x_1,\ldots,x_r)=2^{k_1+\cdots+k_r-r} \sqrt{x_1\cdots x_r}\sum_{\si_1,\ldots,\si_r\in\{\pm 1\}}\si_1 \cdots \si_r \Li_{k_1,\ldots,k_r}(\si_1\sqrt{x_1},\ldots,\si_r\sqrt{x_r}).
\end{align*}
Hence, we obtain
\begin{align*}
&\ti_{k_1,\ldots,k_r}(x_1,\ldots,x_r)-(-1)^{k_1+\cdots+k_r+r}(x_1\cdots x_r)\ti_{k_1,\ldots,k_r}\Big(x_1^{-1},\ldots,x_r^{-1}\Big)\\
&=2^{k_1+\cdots+k_r-r} \sqrt{x_1\cdots x_r}\sum_{\si_1,\ldots,\si_r\in\{\pm 1\}}\si_1 \cdots \si_r \\&\qquad\qquad\quad\times \left(\Li_{k_1,\ldots,k_r}(\si_1\sqrt{x_1},\ldots,\si_r\sqrt{x_r})-(-1)^{k_1+\cdots+k_r+r}\Li_{k_1,\ldots,k_r}\left(\frac1{\si_1\sqrt{x_1}},\ldots,\frac1{\si_r\sqrt{x_r}}\right)\right).
\end{align*}
Then, by further utilizing Panzer's parity theorem regarding cyclotomic multiple polylogarithms, the theorem can be proven.
\end{proof}

\begin{re}
It appears that Panzer's parity theorem cannot determine whether the cyclotomic multiple $t$-values with depth $<r$ in Theorem \ref{thm-weakendedofCMtVparity} will appear, nor can it determine whether the depth can be $\leq N$.
\end{re}

\begin{qu}
Within the ``mod products" of Conjecture \ref{conjparitycmtv}, is it possible that only cyclotomic multiple $t$-values with depth $<r$ and level $\leq N$ appear? That is, cyclotomic multiple zeta values do not appear.
\end{qu}

\begin{qu}
Similar to multiple polylogarithms, can the multiple $t$-polylogarithm function $\ti_{\bfk}(\bfx)$ be analytically continued to the complex plane, yielding a generalization analogous to Panzer's parity theorem for multiple polylogarithms applied to the analytically continued multiple $t$-polylogarithm function?
\end{qu}

The structure of this paper is as follows: In Section 2, we present a residue lemma and the form of the contour integrals to be considered in this paper, as well as the Laurent or Taylor expansions of the integrand in the contour integrals at integer or half-integer points. In Sections 3 and 4, by examining specific contour integrals and computing residues, we derive explicit formulas for the parity of cyclotomic linear and quadratic Euler $T$-sums, thereby establishing parity results for cyclotomic double and triple $t$-values. In Section 5, we utilize the methods of contour integration and residue computation to provide a general theorem on the parity of cyclotomic Euler $T$-sums and present a parity formula for a family cyclotomic cubic Euler $T$-sum. In Section 6, we compute two general formulas for linear Euler $T$-sums involving rational functions. By evaluating specific rational functions, numerous other types of linear Euler $T$-sum results can be obtained. Additionally, we present some parity results for cyclotomic multiple $T$-values and propose several questions and conjectures.

\section{Preliminaries}

Flajolet and Salvy \cite{Flajolet-Salvy} defined a kernel function $\xi(s)$ with two requirements: 1). $\xi(s)$ is meromorphic in the whole complex plane. 2). $\xi(s)$ satisfies $\xi(s)=o(s)$ over an infinite collection of circles $\left| s \right| = {\rho _k}$ with ${\rho _k} \to \infty $. Applying these two conditions of kernel
function $\xi(s)$, Flajolet and Salvy discovered the following residue lemma.

\begin{lem}\emph{(cf.\ \cite{Flajolet-Salvy})}\label{lem-redisue-thm}
Let $\xi(s)$ be a kernel function and let $r(s)$ be a rational function which is $O(s^{-2})$ at infinity. Then
\begin{align}\label{residue-}
\sum\limits_{\alpha  \in O} {{\mathop{\rm Res}}{{\left( {r(s)\xi(s)},\alpha  \right)}}}  + \sum\limits_{\beta  \in S}  {{\mathop{\rm Res}}{{\left( {r(s)\xi(s)}, \beta  \right)}}}  = 0,
\end{align}
where $S$ is the set of poles of $r(s)$ and $O$ is the set of poles of $\xi(s)$ that are not poles $r(s)$. Here ${\mathop{\rm Re}\nolimits} s{\left( {r(s)},\alpha \right)} $ denotes the residue of $r(s)$ at $s= \alpha.$
\end{lem}
Notably, Lemma \ref{lem-redisue-thm} also holds under the weaker condition $r(s)\xi(s)=o(s^{-1})$.

In \cite{Rui-Xu2025}, Rui and Xu defined the \emph{extended trigonometric function} $\phi(s;x)$ and \emph{generalized digamma function} $\Phi(s;x)$ as follows:
\begin{align}
\phi(s;x):=\sum_{k=0}^\infty \frac{x^k}{k+s}\quad (s\notin\N_0^-:=\{0,-1,-2,-3,\ldots\}),
\end{align}
where $x$ is an arbitrary complex number with $|x|\leq 1$ and $x\neq 1$, and
\[\Phi(s;x):=\phi(s;x)-\phi\Big(-s;x^{-1}\Big)-\frac1{s},\]
where to ensure the convergence of the series above, $x$ can only be any root of unity. Clearly, $\phi(s;x)=o(1)$ and $\Phi(s;x)=o(1)$ if $|s|\rightarrow \infty$. In fact, this $\phi(s;x)$ function is a special case of the classical Lerch Zeta Function, and in a recent paper \cite{VH2025}, Vicente and Holgado have studied the Lerch-type zeta function of a recurrence sequence of arbitrary degree.

Rui and Xu provided the Laurent expansions or Maclaurin expansions of functions $\phi(s;x)$ and $\Phi(s;x)$ at integer points.
\begin{lem}(\cite{Rui-Xu2025})\label{lem-rui-xu-one} For $p\in \N$, if $|s+n|<1\ (n\in\N_0:=\N\cup\{0\})$, then
\begin{align}\label{Lexp-phi--n-diffp-1}
\frac{\phi^{(p-1)}(s;x)}{(p-1)!} (-1)^{p-1}&=x^n\sum_{k=0}^\infty \binom{k+p-1}{p-1} \left((-1)^k \Li_{k+p}(x)+(-1)^p\zeta_n\Big(k+p;x^{-1}\Big)\right)(s+n)^k \nonumber\\&\quad+\frac{x^n}{(s+n)^p}\qquad (|s+n|<1,\ n\geq 0)
\end{align}
and
\begin{align}\label{Lexp-phi-n-diffp-1}
\frac{\phi^{(p-1)}(s;x)}{(p-1)!} (-1)^{p-1}&=x^{-n}\sum_{k=0}^\infty \binom{k+p-1}{p-1} (-1)^k  \left( \Li_{k+p}(x)-\zeta_{n-1}\Big(k+p;x\Big)\right)(s-n)^k \nonumber\\&\qquad\qquad\qquad\qquad (|s-n|<1,\ n\geq 1).
\end{align}
\end{lem}
\begin{lem}(\cite{Rui-Xu2025})\label{lem-rui-xu-two} For $n\in \Z$,
\begin{align}\label{LEPhi-function}
\Phi(s;x)=x^{-n} \left(\frac1{s-n}+\sum_{m=0}^\infty \Big((-1)^m\Li_{m+1}(x)-\Li_{m+1}\Big(x^{-1}\Big)\Big)(s-n)^m \right).
\end{align}
\end{lem}

Here, we present the Taylor series expansions of function $\phi(s+1/2;x)$ at integer points and function $\Phi(s;x)$ at half-integer points.

\begin{lem}\label{lem-extend-rui-xu-one} For $p\in \N$, if $|s+n|<1\ (n\geq 0)$, then
\begin{align}\label{Lexp-phi--n-diffp-1}
&\frac{\phi^{(p-1)}(s+1/2;x)}{(p-1)!} (-1)^{p-1}\nonumber\\
&=x^n\sum_{k=0}^\infty \binom{k+p-1}{p-1} \left((-1)^k \ti_{k+p}(x)x^{-1}+(-1)^p t_n\Big(k+p;x^{-1}\Big)\right)(s+n)^k
\end{align}
and if $|s-n|<1\ (n\geq 1)$
\begin{align}\label{Lexp-phi-n-diffp-1}
&\frac{\phi^{(p-1)}(s+1/2;x)}{(p-1)!} (-1)^{p-1}\nonumber\\
&=x^{-n-1}\sum_{k=0}^\infty \binom{k+p-1}{p-1} (-1)^k  \left( \ti_{k+p}(x)-t_{n}\Big(k+p;x\Big)\right)(s-n)^k.
\end{align}
\end{lem}
\begin{proof}
If $|s+n|<1\ (n\geq 0)$, it follows directly from the definition that
\begin{align*}
\phi(s+1/2;x)=x^n \sum_{m=0}^\infty \left((-1)^m \ti_{m+1}(x)x^{-1}-t_n\Big(m+1;x^{-1}\Big) \right)(s+n)^m.
\end{align*}
Taking the $(p-1)$th derivative with respect to $s$ on both sides of the above equation yields formula \eqref{Lexp-phi--n-diffp-1}. Similarly, if $|s-n|<1\ (n\geq 1)$, by a direct calculation, we obtain
\begin{align*}
\phi(s+1/2;x)=x^{-n-1} \sum_{m=0}^\infty (-1)^m \left( \ti_{m+1}(x)-t_n\Big(m+1;x\Big) \right)(s-n)^m.
\end{align*}
Taking the $(p-1)$th derivative with respect to $s$ on both sides of the above equation yields formula \eqref{Lexp-phi-n-diffp-1}.
\end{proof}

\begin{lem}\label{lem-extend-rui-xu-two} If $|s+n+1/2|<1\ (n\geq 0)$, then
\begin{align}
\Phi(s;x)=x^n \sum_{m=0}^\infty \left((-1)^m \ti_{m+1}(x)-x \ti_{m+1}\Big(x^{-1}\Big) \right)(s+n+1/2)^m.
\end{align}
\end{lem}
\begin{proof}
The proof of this lemma is also based on an elementary calculation, which we leave to interested readers to attempt.
\end{proof}

Clearly, on the circle with radius $n+1/2\ (n\in \N)$ centred at the origin, the functions $\phi(s;x)$, $\Phi(s;x)$ and their derivatives are all $O(|s|^\varepsilon)\ (\forall \varepsilon>0)$. Consequently, any
polynomial form in $\Phi(s;x)$ and $\phi^{(j)}(s;x)$ is itself a kernel function with poles at a subset of the integers. Therefore, by applying Lemma \ref{lem-redisue-thm}, we conclude that all contour integrals of the following type vanish:
\begin{align*}
\lim_{R\rightarrow \infty}\oint_{C_R} \frac{\Phi(s;x)\phi^{(p_1-1)}(s+1/2;x_1)\cdots\phi^{(p_r-1)}(s+1/2;x_r)}{(p_1-1)!\cdots (p_r-1)!(s+1/2)^q}(-1)^{p_1+\cdots+p_r-r} ds=0,
\end{align*}
where $p_1,\ldots,p_k,q\in \N$ and $C_R$ denote a circular contour with radius $R$. Hereafter, we shall consistently denote this contour integral limit by $\oint\limits_{\left( \infty  \right)} $.

\section{Parity Results of Linear Cyclotomic Euler $T$-Sums}

In this section, we investigate the parity of linear cyclotomic Euler $T$-sums by constructing contour integrals and performing residue calculations, and provide illustrative examples and corollaries. Furthermore, based on the relationship between cyclotomic linear Euler $T$-sums and cyclotomic double $t$-values, we can derive parity results for cyclotomic double $t$-values.

\begin{thm}\label{thm-linearETS} Let $x,y$ be roots of unity, and $p,q\geq 1$ with $(p,y), (q,xy)\neq (1,1)$. We have
\begin{align}
&x T_{p;q}\Big(y;(xy)^{-1}\Big)-(-1)^{p+q}T_{p;q}\Big(y^{-1};xy\Big)\nonumber\\
&=x \ti_p(y)\ti_q\Big((xy)^{-1}\Big)+(-1)^q y^{-1} \ti_p(y)\ti_q(xy)+(-1)^{p+q-1}\ti_{p+q}(x)\nonumber\\
&\quad+(-1)^q \sum_{m=0}^{p-1}\binom{p+q-m-2}{q-1}\left((-1)^m\ti_{m+1}(x)-x\ti_{m+1}\Big(x^{-1}\Big)\right)\Li_{p+q-m-1}(xy)\nonumber\\
&\quad+(-1)^q \sum_{m=0}^{q-1}\binom{p+q-m-2}{p-1}\left((-1)^m x \ti_{m+1}\Big(x^{-1}\Big)-\ti_{m+1}(x)\right)\Li_{p+q-m-1}(y).
\end{align}
\end{thm}
\begin{proof}
The proof of this theorem is based on residue calculations of the following contour integral:
\begin{align*}
\oint\limits_{\left( \infty  \right)} F_{p,q}(x,y;s)ds:= \oint\limits_{\left( \infty  \right)} \frac{\Phi(s;x)\phi^{(p-1)}(s+1/2;y)}{(p-1)!(s+1/2)^q} (-1)^{p-1}ds=0.
\end{align*}
The integrand $F_{p,q}(x,y;s)$ has the following poles throughout the complex plane: 1. All integers (simple poles); 2. $-1/2$ (pole of order $p+q)$ and 3. $-(n+1/2)$ (for positive integer $n$, poles of order $p$). Applying Lemma \ref{lem-rui-xu-one}-\ref{lem-extend-rui-xu-two}, by direct calculations, we deduce the following residues
\begin{align*}
&\Res\left(F_{p,q}(x,y;s),n\right)=\frac{x^{-n}y^{-n-1}}{(n+1/2)^q}\left(\ti_p(y)-t_n(p;y)\right)\quad (n\geq 0),\\
&\Res\left(F_{p,q}(x,y;s),-n\right)=(-1)^q\frac{(xy)^n}{(n-{1/2})^q}\left(\ti_p(y)y^{-1}+(-1)^p t_n\Big(p;y^{-1}\Big)\right)\quad (n\geq 1),\\
&\Res\left(F_{p,q}(x,y;s),-n-1/2\right)=\frac1{(p-1)!} \lim_{s\rightarrow -n-1/2} \frac{d^{p-1}}{ds^{p-1}}\left((s+n+1/2)^pF_{p,q}(x,y;s)\right)\\
&=(-1)^q\sum_{m=0}^{p-1} \binom{p+q-m-2}{q-1} \left((-1)^m \ti_{m+1}(x)-x \ti_{m+1}\Big(x^{-1}\Big) \right) \frac{(xy)^n}{n^{p+q-m-1}}\quad (n\geq 1)
\end{align*}
and
\begin{align*}
&\Res\left(F_{p,q}(x,y;s),-1/2\right)=\frac1{(p+q-1)!} \lim_{s\rightarrow -1/2} \frac{d^{p+q-1}}{ds^{p+q-1}}\left((s+1/2)^{p+q}F_{p,q}(x,y;s)\right)\\
&=(-1)^{p+q-1}\ti_{p+q}(x)-x \ti_{p+q}\Big(x^{-1}\Big)\\
&\quad+\sum_{m+k=q-1,\atop m,k\geq 0} (-1)^k\binom{k+p-1}{p-1}\Li_{k+p}(y)\left((-1)^m \ti_{m+1}(x)-x\ti_{m+1}\Big(x^{-1}\Big) \right).
\end{align*}
From Lemma \ref{lem-redisue-thm}, we know that
\begin{align*}
&\sum_{n=0}^\infty \Res\left(F_{p,q}(x,y;s),n\right)+\sum_{n=1}^\infty \Res\left(F_{p,q}(x,y;s),-n\right) \\
&\quad+\sum_{n=1}^\infty \Res\left(F_{p,q}(x,y;s),-n-1/2\right) +\Res\left(F_{p,q}(x,y;s),-1/2\right)=0.
\end{align*}
Finally, combining these four contributions yields the statement of Theorem \ref{thm-linearETS}.
\end{proof}

\begin{exa}
Setting $(p,q)=(1,2)$ in Theorem \ref{thm-linearETS} yields
\begin{align*}
&xT_{1;2}\Big(y;(xy)^{-1}\Big)+T_{1;2}\Big(y^{-1};xy\Big)\\&=x\ti_1(y)\ti_2\Big((xy)^{-1}\Big)+y^{-1}\ti_1(y)\ti_2(xy)+\ti_3(x)+\Big(\ti_1(x)-x\ti_1\Big(x^{-1}\Big)\Big)\Li_2(xy)\\&\quad+\Big(x\ti_1\Big(x^{-1}\Big)-\ti_1(x)\Big)\Li_2(y)-\Big(x\ti_2\Big(x^{-1}\Big)+\ti_2(x)\Big)\Li_1(y).
\end{align*}
Setting $(p,q)=(1,3)$ in Theorem \ref{thm-linearETS} yields
\begin{align*}
&xT_{1;3}\Big(y;(xy)^{-1}\Big)-T_{1;3}\Big(y^{-1};xy\Big)\\&=x\ti_1(y)\ti_3\Big((xy)^{-1}\Big)-y^{-1}\ti_1(y)\ti_3(xy)-\ti_4(x)-\Big(\ti_1(x)-x\ti_1\Big(x^{-1}\Big)\Big)\Li_3(xy)\\&\quad-
\Big(x\ti_1\Big(x^{-1}\Big)-\ti_1(x)\Big)\Li_3(y)+\Big(x\ti_2\Big(x^{-1}\Big)+\ti_2(x)\Big)\Li_2(y)\\&\quad-\Big(x\ti_3\Big(x^{-1}\Big)-\ti_3(x)\Big)\Li_1(y).
\end{align*}
Setting $(p,q)=(2,2)$ in Theorem \ref{thm-linearETS} yields
\begin{align*}
&xT_{2;2}\Big(y;(xy)^{-1}\Big)-T_{2;2}\Big(y^{-1};xy\Big)\\&=x\ti_2(y)\ti_2\Big((xy)^{-1}\Big)+y^{-1}\ti_2(y)\ti_2(xy)-\ti_4(x)+2\Big(\ti_1(x)-x\ti_1\Big(x^{-1}\Big)\Big)\Li_3(xy)
\\&\quad-\Big(\ti_2(x)+x\ti_2\Big(x^{-1}\Big)\Big)\Li_2(xy)+2\Big(x\ti_1\Big(x^{-1}\Big)-\ti_1(x)\Big)\Li_3(y)\\&\quad-\Big(x\ti_2\Big(x^{-1}\Big)+\ti_2(x)\Big)\Li_2(y).
\end{align*}
Setting $(p,q)=(3,2)$ in Theorem \ref{thm-linearETS} yields
\begin{align*}
&xT_{3;2}\Big(y;(xy)^{-1}\Big)+T_{3;2}\Big(y^{-1};xy\Big)\\&=x\ti_3(y)\ti_2\Big((xy)^{-1}\Big)+y^{-1}\ti_3(y)\ti_2(xy)+\ti_5(x)+3\Big(\ti_1(x)-x\ti_1\Big(x^{-1}\Big)\Big)\Li_4(xy)\\
&\quad-2\Big(\ti_2(x)+x\ti_2\Big(x^{-1}\Big)\Big)\Li_3(xy)+\Big(\ti_3(x)-x\ti_3\Big(x^{-1}\Big)\Big)\Li_2(xy)\\&\quad+3\Big(x\ti_1\Big(x^{-1}\Big)-\ti_1(x)\Big)\Li_4(y)-\Big(x\ti_2\Big(x^{-1}\Big)+\ti_2(x)\Big)\Li_3(y).
\end{align*}
\end{exa}

Obviously, $t(k_1,k_2,\ldots,k_r)=\ti_{k_1,k_2,\ldots,k_r}(1,1,\ldots,1)$ when $k_r\geq2$. Let $x=y=1$ in Theorem \ref{thm-linearETS}, we have the following corollary.
\begin{cor}
For integers $p,q\geq 2$ with $p+q$ odd, we have
\begin{align*}
2T_{p;q}(1;1)&=t(p)t(q)+(-1)^qt(p)t(q)+t(p+q)\\&\quad-(-1)^q\sum\limits_{k=1}^{\left[p/2\right]}2\binom{p+q-2k-1}{q-1}t(2k)\zeta(p+q-2k)
\\&\quad-(-1)^q\sum\limits_{k=1}^{\left[q/2\right]}2\binom{p+q-2k-1}{p-1}t(2k)\zeta(p+q-2k).
\end{align*}
\end{cor}

\begin{exa}
Since $t(i)=(2^i-1)\zeta(i)$, we have
    \begin{align*}
    &T_{2;3}(1;1)=\frac{1}{2}t(5)+\frac{3}{7}t(2)t(3);\\
    &T_{3;2}(1;1)=\frac{1}{2}t(5)+\frac{4}{7}t(2)t(3);\\
    &T_{3;4}(1;1)=\frac{1}{2}t(7)+\frac{6}{7}t(3)t(4)-\frac{10}{31}t(2)t(5);\\
    &T_{4;3}(1;1)=\frac{1}{2}t(7)+\frac{1}{7}t(3)t(4)+\frac{10}{31}t(2)t(5);\\
    &T_{2;5}(1;1)=\frac{1}{2}t(7)+\frac{5}{31}t(2)t(5)+\frac{2}{7}t(3)t(4);\\
    &T_{5;2}(1;1)=\frac{1}{2}t(7)+\frac{26}{31}t(2)t(5)-\frac{2}{7}t(3)t(4).
    \end{align*}
\end{exa}
Finally, according to definition of cyclotomic linear Euler $T$-sums and cyclotomic double $t$-values, we have
\[T_{p;q}(x;y)=\ti_{p,q}(x,y)+\ti_{p+q}(xy).\]
Therefore, we can derive the following corollary regarding the parity of cyclotomic double $t$-values.
\begin{cor}\label{cordoublecmtv}
Let $x$ and $y$ be N-th roots of unity, and $p,q\geq 1$ with $(p,y), (q,xy)\neq (1,1)$. Then
\[x \ti_{p,q}\Big(y,(xy)^{-1}\Big)-(-1)^{p+q}\ti_{p,q}\Big(y^{-1},xy\Big)\]
reduces to a combination of cyclotomic single $t$-values and cyclotomic single zeta values with level $\leq N$.
\end{cor}

\section{Parity Results of Quadratic Cyclotomic Euler $T$-Sums}
In this section, we employ the method of contour integration to derive the parity formulas for cyclotomic quadratic Euler $T$-sums and further present parity results for depth-three cyclotomic multiple $t$-values.

\begin{thm}
\label{thm-quadratic-C-ES-one} Let $x,x_1,x_2$ be roots of unity, and $p_1,p_2,q\geq 1$ with $(p_1,x_1), (p_2,x_2) $ and $ (q,xx_1x_2)\neq (1,1)$. We have
\begin{align}
&x T_{p_1,p_2;q}\Big(x_1,x_2;(xx_1x_2)^{-1}\Big)+(-1)^{p_1+p_2+q} T_{p_1,p_2;q}\Big(x_1^{-1},x_2^{-1};xx_1x_2\Big)\nonumber\\
&=x T_{p_1;p_2+q}\Big(x_1;(xx_1)^{-1}\Big)+x T_{p_2;p_1+q}\Big(x_2;(xx_2)^{-1}\Big)\nonumber\\
&\quad+x\ti_{p_1}(x_1)T_{p_2;q}\Big(x_2;(xx_1x_2)^{-1}\Big)+x\ti_{p_2}(x_2)T_{p_1;q}\Big(x_1;(xx_1x_2)^{-1}\Big)\nonumber\\
&\quad-(-1)^{p_2+q}x_1^{-1}\ti_{p_1}(x_1)T_{p_2;q}\Big(x_2^{-1};xx_1x_2\Big)-(-1)^{p_1+q}x_2^{-1}\ti_{p_2}(x_2)T_{p_1;q}\Big(x_1^{-1};xx_1x_2\Big)\nonumber\\
&\quad+(-1)^{p_1+p_2+q}\ti_{p_1+p_2+q}(x)-x \ti_{p_1}(x_1)\ti_{p_2+q}\Big((xx_1)^{-1}\Big)\nonumber\\&\quad-x \ti_{p_2}(x_2)\ti_{p_1+q}\Big((xx_2)^{-1}\Big)-x\ti_{p_1}(x_1)\ti_{p_2}(x_2)\ti_q\Big((xx_1x_2)^{-1}\Big)\nonumber\\
&\quad-(-1)^q(x_1x_2)^{-1}\ti_{p_1}(x_1)\ti_{p_2}(x_2)\ti_q(xx_1x_2)\nonumber\\
&\quad-\sum_{m+k=p_1+q-1,\atop m,k\geq 0} (-1)^k \binom{k+p_2-1}{p_2-1}\Li_{k+p_2}(x_2)\left((-1)^m\ti_{m+1}(x)-x \ti_{m+1}\Big(x^{-1}\Big)\right)\nonumber\\
&\quad-\sum_{m+k=p_2+q-1,\atop m,k\geq 0} (-1)^k \binom{k+p_1-1}{p_1-1}\Li_{k+p_1}(x_1)\left((-1)^m\ti_{m+1}(x)-x \ti_{m+1}\Big(x^{-1}\Big)\right)\nonumber\\
&\quad-(-1)^q\sum_{m=0}^{p_1+p_2-1}\binom{p_1+p_2+q-m-2}{q-1}\left((-1)^m\ti_{m+1}(x)-x \ti_{m+1}\Big(x^{-1}\Big)\right)\nonumber\\&\qquad\qquad\qquad\qquad\qquad\times \Li_{p_1+p_2+q-m-1}(xx_1x_2)\nonumber\\
&\quad-\sum_{m+k_1+k_2=q-1,\atop m,k_1,k_2\geq 0}(-1)^{k_1+k_2}\binom{k_1+p_1-1}{p_1-1}\binom{k_2+p_2-1}{p_2-1}\Li_{k_1+p_1}(x_1)\Li_{k_2+p_2}(x_2)\nonumber\\&\qquad\qquad\qquad\qquad\qquad\times \left((-1)^m\ti_{m+1}(x)-x \ti_{m+1}\Big(x^{-1}\Big)\right)\nonumber\\
&\quad-(-1)^q \sum_{m+k\leq p_2-1,\atop m,k\geq 0} \binom{k+p_1-1}{p_1-1}\binom{p_2+q-m-k-2}{q-1}\left((-1)^m\ti_{m+1}(x)-x \ti_{m+1}\Big(x^{-1}\Big)\right)\nonumber\\&\quad\times\left((-1)^k\Li_{k+p_1}(x_1)\Li_{p_2+q-m-k-1}(xx_1x_2)+(-1)^{p_1}S_{k+p_1;p_2+q-m-k-1}\Big(x_1^{-1};xx_1x_2\Big)\right)\nonumber\\
&\quad-(-1)^q \sum_{m+k\leq p_1-1,\atop m,k\geq 0} \binom{k+p_2-1}{p_2-1}\binom{p_1+q-m-k-2}{q-1} \left((-1)^m\ti_{m+1}(x)-x \ti_{m+1}\Big(x^{-1}\Big)\right)\nonumber\\&\quad\times\left((-1)^k\Li_{k+p_2}(x_2)\Li_{p_1+q-m-k-1}(xx_1x_2)+(-1)^{p_2}S_{k+p_2;p_1+q-m-k-1}\Big(x_2^{-1};xx_1x_2\Big)\right).
\end{align}
\end{thm}

\begin{proof}
The proof of this theorem is based on residue calculations of the following contour integral:
\[\oint\limits_{\left( \infty  \right)}F_{p_1p_2,q}(s)ds:= \oint\limits_{\left( \infty  \right)}\frac{\Phi(s;x)\phi^{(p_1-1)}(s+1/2;x_1)\phi^{(p_2-1)}(s+1/2;x_2)}{(p_1-1)!(p_2-1)!(s+1/2)^q} (-1)^{p_1+p_2}ds=0.\]
Clearly, the integrand $F^{(a)}_{p_1p_2,q}(x,x_1,x_2;s)$ possesses the following poles in the complex plane: 1. All integer points are simple poles; 2. $s=-1/2$ is a pole of order $p_1+p_2+q$; 3. $s=-n-1/2$ (where $n$ is a positive integer) is a pole of order $p_1+p_2$. Applying Lemmas \ref{lem-rui-xu-two} and \ref{lem-extend-rui-xu-one}, through a direct computation, we obtain the residue values at integer points as follows:
\begin{align*}
&\Res\left(F_{p_1p_2,q}(\cdot;s),n\right)=\frac{x^{-n}(x_1x_2)^{-n-1}}{(n+1/2)^q}\left(\ti_{p_1}(x_1)-t_n(p_1;x_1)\right)
\\&\qquad\qquad\qquad\qquad\qquad\qquad\qquad\times\left(\ti_{p_2}(x_2)-t_n(p_2;x_2)\right)\quad (n\in \N_0),\\
&\Res\left(F_{p_1p_2,q}(\cdot;s),-n\right)=(-1)^q \frac{(xx_1x_2)^n}{(n-1/2)^q}\left(\ti_{p_1}(x_1)x_1^{-1}+(-1)^{p_1}t_n\Big(p_1;x_1^{-1}\Big)\right)\\&\qquad\qquad\qquad\qquad\qquad\qquad\qquad
\times\left(\ti_{p_2}(x_2)x_2^{-1}+(-1)^{p_2}t_n\Big(p_2;x_2^{-1}\Big)\right)\quad (n\in \N).
\end{align*}
Applying Lemmas \ref{lem-rui-xu-one} and \ref{lem-extend-rui-xu-two}, after lengthy calculations, we obtain
\begin{align*}
&\Res\left(F_{p_1p_2,q}(\cdot;s),-1/2\right)\\&=\frac{1}{(p_1+p_2+q-1)!}\lim_{s\rightarrow -1/2}\frac{d^{p_1+p_2+q-1}}{ds^{p_1+p_2+q-1}}\left\{(s+1/2)^{p_1+p_2+q}F_{p_1p_2,q}(x,x_1,x_2;s)\right\}\\
&=(-1)^{p_1+p_2+q-1}\ti_{p_1+p_2+q}(x)-x\ti_{p_1+p_2+q}\Big(x^{-1}\Big)\\
&\quad+\sum_{m+k=p_1+q-1,\atop m,k\geq 0} (-1)^k \binom{k+p_2-1}{p_2-1}\Li_{k+p_2}(x_2)\left((-1)^m\ti_{m+1}(x)-x \ti_{m+1}\Big(x^{-1}\Big)\right)\\
&\quad+\sum_{m+k=p_2+q-1,\atop m,k\geq 0} (-1)^k \binom{k+p_1-1}{p_1-1}\Li_{k+p_1}(x_1)\left((-1)^m\ti_{m+1}(x)-x \ti_{m+1}\Big(x^{-1}\Big)\right)\\
&\quad+\sum_{m+k_1+k_2=q-1,\atop m,k_1,k_2\geq 0}(-1)^{k_1+k_2}\binom{k_1+p_1-1}{p_1-1}\binom{k_2+p_2-1}{p_2-1}\Li_{k_1+p_1}(x_1)\Li_{k_2+p_2}(x_2)\nonumber\\&\qquad\qquad\qquad\qquad\qquad\times \left((-1)^m\ti_{m+1}(x)-x \ti_{m+1}\Big(x^{-1}\Big)\right)
\end{align*}
and for $n\in\N$,
\begin{align*}
&\Res\left(F_{p_1p_2,q}(\cdot;s),-n-1/2\right)\\&=\frac{1}{(p_1+p_2-1)!}\lim_{s\rightarrow -n-1/2}\frac{d^{p_1+p_2-1}}{ds^{p_1+p_2-1}}\left\{(s+n+1/2)^{p_1+p_2}F_{p_1p_2,q}(x,x_1,x_2;s)\right\}\\
&=(-1)^q\sum_{m=0}^{p_1+p_2-1}\binom{p_1+p_2+q-m-2}{q-1}\left((-1)^m\ti_{m+1}(x)-x \ti_{m+1}\Big(x^{-1}\Big)\right)\\&\qquad\qquad\qquad\qquad\qquad\times \frac{(xx_1x_2)^n}{n^{p_1+p_2+q-m-1}}\\
&\quad+(-1)^q \sum_{m+k\leq p_2-1,\atop m,k\geq 0} \binom{k+p_1-1}{p_1-1}\binom{p_2+q-m-k-2}{q-1}\frac{(xx_1x_2)^n}{n^{p_2+q-m-k-1}} \nonumber\\&\quad\times\left((-1)^m\ti_{m+1}(x)-x \ti_{m+1}\Big(x^{-1}\Big)\right)\left((-1)^k\Li_{k+p_1}(x_1)+(-1)^{p_1}\zeta_n\Big(k+p_1;x_1^{-1}\Big)\right)\\
&\quad+(-1)^q \sum_{m+k\leq p_1-1,\atop m,k\geq 0} \binom{k+p_2-1}{p_2-1}\binom{p_1+q-m-k-2}{q-1}\frac{(xx_1x_2)^n}{n^{p_1+q-m-k-1}} \nonumber\\&\quad\times\left((-1)^m\ti_{m+1}(x)-x \ti_{m+1}\Big(x^{-1}\Big)\right)\left((-1)^k\Li_{k+p_2}(x_1)+(-1)^{p_2}\zeta_n\Big(k+p_2;x_2^{-1}\Big)\right).
\end{align*}
By Lemma \ref{lem-redisue-thm}, we have
\begin{align*}
&\sum_{n=0}^\infty \Res\left(F_{p_1p_2,q}(\cdot;s),n\right)+\sum_{n=1}^\infty \Res\left(F_{p_1p_2,q}(\cdot;s),-n\right)\\&\quad+\sum_{n=1}^\infty\Res\left(F_{p_1p_2,q}(\cdot;s),-n-1/2\right)+\Res\left(F_{p_1p_2,q}(\cdot;s),-1/2\right)=0.
\end{align*}
Substituting the four residue results obtained above consequently proves Theorem \ref{thm-quadratic-C-ES-one}.
\end{proof}

\begin{exa}
Setting $(p_1,p_2,q)=(1,1,2)$ in Theorem \ref{thm-quadratic-C-ES-one}, we have
\begin{align*}
&xT_{1,1;2}\Big(x_1,x_2;(xx_1x_2)^{-1}\Big)+T_{1,1;2}\Big(x_1^{-1},x_2^{-1};xx_1x_2\Big)
\\&=-\Big(\ti_3(x)-x\ti_3\Big(x^{-1}\Big)\Big)\Big(\Li_1(x_1)+\Li_1(x_2)\Big)
\\&\quad-\Big(\ti_2(x)+x\ti_2\Big(x^{-1}\Big)\Big)\Big(\Li_2(x_1)+\Li_2(x_2)-\Li_2(xx_1x_2)-\Li_1(x_1)\Li_1(x_2)\Big)
\\&\quad-\Big(\ti_1(x)-x\ti_1\Big(x^{-1}\Big)\Big)\Big\{\Li_3(x_1)+\Li_3(x_2)+2\Li_3(xx_1x_2)-\Li_2(x_1)\Li_1(x_2)-\Li_1(x_1)\Li_2(x_2)\\
&\qquad\qquad+\Li_1(x_1)\Li_2(xx_1x_2)+\Li_1(x_2)\Li_2(xx_1x_2)-S_{1;2}\Big(x_1^{-1};xx_1x_2\Big)-S_{1;2}\Big(x_2^{-1};xx_1x_2\Big)\Big\}
\\&\quad+xT_{1;3}\Big(x_1;(xx_1)^{-1}\Big)+xT_{1;3}\Big(x_2;(xx_2)^{-1}\Big)+x\ti_1(x_1)T_{1;2}\Big(x_2;(xx_1x_2)^{-1}\Big)
\\&\quad+x\ti_1(x_2)T_{1;2}\Big(x_1;(xx_1x_2)^{-1}\Big)+x_1^{-1}\ti_1(x_1)T_{1;2}\Big(x_2^{-1};xx_1x_2\Big)+x_2^{-1}\ti_1(x_2)T_{1;2}\Big(x_1^{-1};xx_1x_2\Big)
\\&\quad+\ti_4(x)-x\ti_1(x_1)\ti_3\Big(\Big(xx_1\Big)^{-1}\Big)-x\ti_1(x_2)\ti_3\Big((xx_2)^{-1}\Big)-x\ti_1(x_1)\ti_1(x_2)\ti_2\Big((xx_1x_2)^{-1}\Big)
\\&\quad-(x_1x_2)^{-1}\ti_1(x_1)\ti_1(x_2)\ti_2(xx_1x_2).
\end{align*}
Setting $(p_1,p_2,q)=(1,2,2)$ in Theorem \ref{thm-quadratic-C-ES-one}, we have
\begin{align*}
&xT_{1,2;2}\Big(x_1,x_2;(xx_1x_2)^{-1}\Big)-T_{1,2;2}\Big(x_1^{-1},x_2^{-1};xx_1x_2\Big)\\
&=xT_{1;4}\Big(x_1;(xx_1)^{-1}\Big)+xT_{2;3}\Big(x_2;(xx_2)^{-1}\Big)+x\ti_1(x_1)T_{2;2}\Big(x_2;(xx_1x_2)^{-1}\Big)
\\&\quad+x\ti_2(x_2)T_{1;2}\Big(x_1;(xx_1x_2)^{-1}\Big)-x_1^{-1}\ti_1(x_1)T_{2;2}\Big(x_2^{-1};xx_1x_2\Big)+x_2^{-1}\ti_2(x_2)T_{1;2}\Big(x_1^{-1};xx_1x_2\Big)
\\&\quad-\ti_5(x)-x\ti_1(x_1)\ti_4\Big((xx_1)^{-1}\Big)-x\ti_2(x_2)\ti_3\Big((xx_2)^{-1}\Big)-x\ti_1(x_1)\ti_2(x_2)\ti_2\Big((xx_1x_2)^{-1}\Big)
\\&\quad-(x_1x_2)^{-1}\ti_1(x_1)\ti_2(x_2)\ti_2(xx_1x_2)+\Big(\ti_4(x)+x\ti_4\Big(x^{-1}\Big)\Big)\Li_1(x_1)\\&\quad
+\Big(\ti_3(x)-x\ti_3\Big(x^{-1}\Big)\Big)\Big(\Li_2(x_1)-\Li_2(x_2)-\Li_2(xx_1x_2)\Big)
\\&\quad+\Big(\ti_2(x)+x\ti_2\Big(x^{-1}\Big)\Big)\Big(\Li_3(x_1)-2\Li_3(x_2)+2\Li_3(xx_1x_2)+\Li_1(x_1)\Li_2(x_2)
\\&\qquad\qquad\qquad\qquad\qquad\qquad\qquad+\Li_1(x_1)\Li_2(xx_1x_2)-S_{1;2}\Big(x_1^{-1};xx_1x_2\Big)\Big)
\\&\quad+\Big(\ti_1(x)-x\ti_1\Big(x^{-1}\Big)\Big)\Big\{\Li_4(x_1)-3\Li_4(x_2)-3\Li_4(xx_1x_2)+\Li_2(x_1)\Li_2(x_2)
\\&\quad\qquad+2\Li_1(x_1)\Li_3(x_2)-2\Li_1(x_1)\Li_3(xx_1x_2)+\Li_2(x_1)\Li_2(xx_1x_2)-\Li_2(x_2)\Li_2(xx_1x_2)
\\&\quad\qquad\qquad\qquad\qquad\qquad\qquad+2S_{1;3}\Big(x_1^{-1};xx_1x_2\Big)+S_{2;2}\Big(x_1^{-1};xx_1x_2\Big)-S_{2;2}\Big(x_2^{-1};xx_1x_2\Big)\Big\}.
\end{align*}
\end{exa}

Finally, according to definition of cyclotomic quadratic Euler $T$-sums and cyclotomic triple $t$-values, for $(p_1,x_1),(q,x)\neq (1,1)$, we have
\begin{align*}
T_{p_1,p_2;q}(x_1,x_2;x)&=\sum_{n=1}^\infty \frac{t_n(p_1;x_1)t_n(p_2;x_2)}{(n-1/2)^q}x^n\\
&=\sum_{n=1}^\infty \frac{\left(t_n(p_1;x_1)-\ti_{p_1}(x_1)\right)t_n(p_2;x_2)}{(n-1/2)^q}x^n+\ti_{p_1}(x_1) \sum_{n=1}^\infty \frac{t_n(p_2;x_2)}{(n-1/2)^q}x^n\\
&=\sum_{n=1}^\infty \frac{\left(t_n(p_1;x_1)-\ti_{p_1}(x_1)\right)t_{n-1}(p_2;x_2)}{(n-1/2)^q}x^n+\sum_{n=1}^\infty \frac{t_n(p_1;x_1)-\ti_{p_1}(x_1)}{(n-1/2)^{p_2+q}}(x_2x)^n\\
&\quad+\ti_{p_1}(x_1) \sum_{n=1}^\infty \frac{\left(t_{n-1}(p_2;x_2)+\frac{x_2^n}{(n-1/2)^{p_2}}\right)}{(n-1/2)^q}x^n\\
&=-\ti_{p_2,q,p_1}(x_2,x,x_1)-\ti_{p_2+q,p_1}(x_2x,x_1)+\ti_{p_1}(x_1)\left(\ti_{p_2,q}(x_2,x)+\ti_{p_2+q}(x_2x)\right).
\end{align*}
Therefore, we can derive the following corollary regarding the parity of cyclotomic triple $t$-values with a direct calculation.
\begin{cor}\label{cor-quadratic-C-ES-one}
Let $x,y,z$ be $N$th-roots of unity, and $p,m,q\geq 1$ with $(p,x), (q,y)$ and $ (m,z)\neq (1,1)$. Then
\begin{align*}
\ti_{p,q,m}(x,y,z)+(-1)^{p+q+m}xyz\ti_{p,q,m}\Big(x^{-1},y^{-1},z^{-1}\Big)
\end{align*}
reduces to a combination of cyclotomic double zeta values, cyclotomic double $t$-values, cyclotomic single $t$-values and cyclotomic single zeta values with level $\leq N$.
\end{cor}

\begin{exa}
 Let $(p,q,m)=(1,2,1)$ in Corollary \ref{cor-quadratic-C-ES-one}, we have
\begin{align*}
 &\ti_{1,2,1}(x,y,z)+xyz\ti_{1,2,1}\Big(x^{-1},y^{-1},z^{-1}\Big)
 \\&=\Big(xyz\ti_3\Big((xyz)^{-1}\Big)-\ti_3(xyz)\Big)\Big\{\Li_1(x)+\Li_1(z)\Big\}
\\&\quad+\Big(xyz\ti_2\Big((xyz)^{-1}\Big)+\ti_2(xyz)\Big)\Big\{\Li_2(x)+\Li_2(z)-\Li_2\Big(y^{-1}\Big)-\Li_1(x)\Li_1(z)\Big\}
 \\&\quad+\Big(xyz\ti_1\Big((xyz)^{-1}\Big)-\ti_1(xyz)\Big)\Big\{\Li_3(x)+\Li_3(z)+2\Li_3\Big(y^{-1}\Big)-\Li_1(x)\Li_2(z)-\Li_1(z)\Li_2(x)\\
 &\qquad\qquad\qquad\qquad\qquad+\Li_1(z)\Li_2\Big(y^{-1}\Big)+\Li_1(x)\Li_2\Big(y^{-1}\Big)-S_{1;2}\Big(z^{-1};y^{-1}\Big)-S_{1;2}\Big(x^{-1};y^{-1}\Big)\Big\}
 \\&\quad-\ti_{1,3}(z,xy)-\ti_{1,3}(x,yz)-\ti_{3,1}(xy,z)-2\ti_4(xyz)+\ti_1(z)\ti_{2,1}(x,y)-\ti_1(z)\ti_{1,2}(x,y)
 \\&\quad-\ti_1(x)\ti_{1,2}(z,y)-xy\ti_1(z)\ti_{1,2}\Big(x^{-1},y^{-1}\Big)+\ti_1(z)\ti_{3}(xy)-xy\ti_1(z)\ti_{3}\Big((xy)^{-1}\Big)
 \\&\quad- yz\ti_1(x)\ti_{1,2}\Big(z^{-1},y^{-1}\Big)-yz\ti_1(x)\ti_{3}\Big((yz)^{-1}\Big)-xyz\ti_4\Big((xyz)^{-1}\Big)+\ti_1(x)\ti_1(z)\ti_2(y)
 \\&\quad+y\ti_1(x)\ti_1(z)\ti_2\Big(y^{-1}\Big)-xyz\ti_{3,1}\Big((xy)^{-1},z^{-1}\Big)+xyz\ti_1\Big(z^{-1}\Big)\ti_{2,1}\Big(x^{-1},y^{-1}\Big)\\&\quad+xyz\ti_1\Big(z^{-1}\Big)\ti_{3}\Big((xy)^{-1}\Big).
\end{align*}
 Let $(p,q,m)=(2,2,2)$ in Corollary \ref{cor-quadratic-C-ES-one}, we have
\begin{align*}
 &\ti_{2,2,2}(x,y,z)+xyz\ti_{2,2,2}\Big(x^{-1},y^{-1},z^{-1}\Big)
 \\&=-\Big(xyz\ti_4\Big((xyz)^{-1}\Big)+\ti_4(xyz)\Big)\Big\{\Li_2(x)+\Li_2(z)+\Li_2\Big(y^{-1}\Big)\Big\}
\\&\quad-\Big(xyz\ti_3\Big((xyz)^{-1}\Big)-\ti_3(xyz)\Big)\Big\{2\Li_3(x)+2\Li_3(z)-2\Li_3\Big(y^{-1}\Big)\Big\}
\\&\quad-\Big(xyz\ti_2\Big((xyz)^{-1}\Big)+\ti_2(xyz)\Big)\Big\{3\Li_4(x)+3\Li_4(z)+3\Li_4\Big(y^{-1}\Big)+\Li_2(z)\Li_2(x)
\\&\qquad\qquad\qquad\qquad+\Li_2(z)\Li_2\Big(y^{-1}\Big)+\Li_2(x)\Li_2\Big(y^{-1}\Big)+S_{2;2}\Big(z^{-1};y^{-1}\Big)+S_{2;2}\Big(x^{-1};y^{-1}\Big)\Big\}
 \\&\quad-\Big(xyz\ti_1\Big((xyz)^{-1}\Big)-\ti_1(xyz)\Big)\Big\{4\Li_5(x)+4\Li_5(z)-4\Li_5\Big(y^{-1}\Big)+2\Li_3(z)\Li_2(x)
\\&\qquad\qquad\qquad+2\Li_2(z)\Li_3(x)-2\Li_2(z)\Li_3\Big(y^{-1}\Big)+2\Li_3(z)\Li_2\Big(y^{-1}\Big)-2\Li_2(x)\Li_3\Big(y^{-1}\Big)
\\&\qquad\qquad\qquad+2\Li_3(x)\Li_2\Big(y^{-1}\Big)-2S_{2;3}\Big(z^{-1};y^{-1}\Big)-2S_{3;2}\Big(z^{-1};y^{-1}\Big)-2S_{2;3}\Big(x^{-1};y^{-1}\Big)
\\&\qquad\qquad\qquad-2S_{3;2}\Big(x^{-1};y^{-1}\Big)\Big\}
 \\&\quad-\ti_{4,2}(xy,z)-\ti_{2,4}(x,yz)+\ti_2(z)\ti_4(xy)-xyz\ti_{4,2}\Big((xy)^{-1},z^{-1}\Big)
 \\&\quad+xyz\ti_2\Big(z^{-1}\Big)\ti_{2,2}\Big(x^{-1},y^{-1}\Big)+xyz\ti_2\Big(z^{-1}\Big)\ti_4\Big((xy)^{-1}\Big)-\ti_{2,4}(z,xy)-2\ti_6(xyz)
 \\&\quad-\ti_2(x)\ti_{2,2}(z,y)+xy\ti_2(z)\ti_{2,2}\Big(x^{-1},y^{-1}\Big)+xy\ti_2(z)\ti_4\Big((xy)^{-1}\Big)
 \\&\quad+yz\ti_2(x)\ti_{2,2}\Big(z^{-1},y^{-1}\Big)+yz\ti_2(x)\ti_4\Big((yz)^{-1}\Big)-xyz\ti_6\Big((xyz)^{-1}\Big)
 \\&\quad+\ti_2(x)\ti_2(y)\ti_2(z)+y\ti_2(x)\ti_2\Big(y^{-1}\Big)\ti_2(z).
\end{align*}
\end{exa}

\begin{exa}
    Considering $(p,q,m)=(2,2,2)$ and $x=x_1=x_2=1$ in Corollary \ref{cor-quadratic-C-ES-one}, we have
    \begin{align*}
    &t(2,2,2)=\frac{5}{9}t(2)t(2)t(2)+t(2)t(2,2)+\frac{1}{3}t(2)t(4)-t(2,4)-t(4,2)-\frac{3}{2}t(6).
    \end{align*}
    Note that $t(2)t(2)=\frac{3}{2}t(4)$ (see \cite{H2019}), by using the stuffle relations among multiple $t$-values, we obtain
    \begin{align*}
    t(2,2,2)=\frac{1}{48}t(6).
    \end{align*}
\end{exa}

\section{Parity Results of Generalized Cyclotomic Euler $T$-Sums}
In this section, we utilize the method of contour integration to present parity results and several examples for cyclotomic Euler $T$-sums of arbitrary order.
\begin{thm}\label{thm-parityc-CES-one} Let $x,x_1,\ldots,x_r$ be roots of unity, and $p_1,\ldots,p_r,q\geq 1$ with $(p_j,x_j) $ and $ (q,xx_1\cdots x_r)\neq (1,1)$. The
\begin{align*}
&x T_{p_1,p_2,\ldots,p_r;q}\Big(x_1,x_2,\ldots,x_r;(xx_1\cdots x_r)^{-1}\Big)\\&\quad+(-1)^{p_1+p_2+\cdots+p_r+q+r}T_{p_1,p_2,\ldots,p_r;q}\Big(x_1^{-1},x_2^{-1},\ldots,x_r^{-1};xx_1\cdots x_r\Big)
\end{align*}
reduces to a combination of sums of lower orders (It should be emphasized that the lower-order sums include not only cyclotomic Euler $T$-sums but also cyclotomic Euler sums).
\end{thm}
\begin{proof}
The proof of this theorem is based on residue calculations of the following contour integral:
\begin{align*}
&\oint\limits_{\left( \infty  \right)}F_{p_1p_2\cdots p_r,q}(x,x_1,x_2,\cdots,x_r;s)ds\\
&:= \oint\limits_{\left( \infty  \right)}\frac{\Phi(s;x)\phi^{(p_1-1)}(s+1/2;x_1)\cdots\phi^{(p_r-1)}(s+1/2;x_r)}{(p_1-1)!\cdots (p_r-1)!(s+1/2)^q} (-1)^{p_1+\cdots+p_r-r}ds=0.
\end{align*}
Obviously, the integrand $F_{p_1p_2\cdots p_r,q}(x,x_1,x_2,\ldots,x_r;s)$ possesses the following poles in the complex plane: 1. All integer points are simple poles; 2. $s=-1/2$ is a pole of order $p_1+p_2+\cdots+p_r+q$; 3. $s=-n-1/2$ (where $n$ is a positive integer) is a pole of order $p_1+p_2+\cdots+p_r$. Applying Lemma \ref{lem-redisue-thm}, we have
\begin{align}\label{residue-sums-quad}
&\sum_{n=0}^\infty \Res\left(F_{p_1p_2\cdots p_r,q}(\cdot;s),n\right)+\sum_{n=1}^\infty \Res\left(F_{p_1p_2\cdots p_r,q}(\cdot;s),-n\right)\nonumber\\&\quad+\sum_{n=1}^\infty\Res\left(F_{p_1p_2\cdots p_r,q}(\cdot;s),-n-1/2\right)+\Res\left(F_{p_1p_2\cdots p_r,q}(\cdot;s),-1/2\right)=0.
\end{align}
At integer points, which are simple zeros, the residue values can be calculated using Lemmas \ref{lem-rui-xu-two} and \ref{lem-extend-rui-xu-one} as follows:
\begin{align*}
&\Res\left(F_{p_1p_2\cdots p_r,q}(\cdot;s),n\right)=\frac{x^{-n}(x_1\cdots x_r)^{-n-1}}{(n+1/2)^q}\prod\limits_{j=1}^r\left(\ti_{p_j}(x_j)-t_n(p_j;x_j)\right)\quad (n\in \N_0),\\
&\Res\left(F_{p_1p_2\cdots p_r,q}(\cdot;s),-n\right)=\frac{(xx_1\cdots x_r)^n}{(-n+1/2)^q}\prod\limits_{j=1}^r\left(\ti_{p_j}(x_j)x_j^{-1}+(-1)^{p_j}t_n\Big(p_j;x_j^{-1}\Big)\right)\quad (n\in \N).
\end{align*}
By expanding the two residue values above and then summing them, we obtain
\begin{align*}
&\sum_{n=0}^\infty \Res\left(F_{p_1p_2\cdots p_r,q}(\cdot;s),n\right)+\sum_{n=1}^\infty \Res\left(F_{p_1p_2\cdots p_r,q}(\cdot;s),-n\right)\\
&=x T_{p_1,p_2,\ldots,p_r;q}\Big(x_1,x_2,\cdots,x_r;(xx_1\cdots x_r)^{-1}\Big)\\&\quad+(-1)^{p_1+p_2+\cdots+p_r+q+r}T_{p_1,p_2,\ldots,p_r;q}\Big(x_1^{-1},x_2^{-1},\ldots,x_r^{-1};xx_1\cdots x_r\Big)\\
&\quad+\{\text{combinations of lower-order sums}\}.
\end{align*}
Applying Lemmas \ref{lem-rui-xu-one} and \ref{lem-extend-rui-xu-two}, we can also compute the latter two residue values in \eqref{residue-sums-quad}. However, the resulting sum obtained after summation will still be of order less than $r$, namely:
\begin{align*}
&\sum_{n=1}^\infty\Res\left(F_{p_1p_2\cdots p_r,q}(\cdot;s),-n-1/2\right)+\Res\left(F_{p_1p_2\cdots p_r,q}(\cdot;s),-1/2\right)\\
&\in \{\text{combinations of lower-order sums}\}.
\end{align*}
Finally, substituting these two conclusions into \eqref{residue-sums-quad} completes the proof of the theorem.
\end{proof}

\begin{exa}
As an example, considering $p_1=p_2=p_3=1$ in Theorem \ref{thm-parityc-CES-one}, we have
\begin{align*}
&xT_{1^3;q}\Big(x_1,x_2,x_3;(xx_1x_2x_3)^{-1}\Big)+(-1)^qT_{1^3;q}\Big(x_1^{-1},x_2^{-1},x_3^{-1};xx_1x_2x_3\Big)
\\=&x\ti_1(x_1)\ti_1(x_2)\ti_1(x_3)\ti_q\Big((xx_1x_2x_3)^{-1}\Big)\\&-x\ti_1(x_1)\ti_1(x_2)\Big(T_{1;q}\Big(x_3;(xx_1x_2x_3)^{-1}\Big)-\ti_{q+1}\Big((xx_1x_2)^{-1}\Big)\Big)
\\&-x\ti_1(x_1)\ti_1(x_3)\Big(T_{1;q}\Big(x_2;(xx_1x_2x_3)^{-1}\Big)-\ti_{q+1}\Big((xx_1x_3)^{-1}\Big)\Big)
\\&-x\ti_1(x_2)\ti_1(x_3)\Big(T_{1;q}\Big(x_1;(xx_1x_2x_3)^{-1}\Big)-\ti_{q+1}\Big((xx_2x_3)^{-1}\Big)\Big)
\\&+x\ti_1(x_1)\Big(T_{1,1;q}\Big(x_2,x_3;(xx_1x_2x_3)^{-1}\Big)-T_{1;q+1}\Big(x_2;(xx_1x_2)^{-1}\Big)\\&-T_{1;q+1}\Big(x_3;(xx_1x_3)^{-1}\Big)+\ti_{q+2}\Big((xx_1)^{-1}\Big)\Big)
\\&+x\ti_1(x_2)\Big(T_{1,1;q}\Big(x_1,x_3;(xx_1x_2x_3)^{-1}\Big)-T_{1;q+1}\Big(x_1;(xx_1x_2)^{-1}\Big)\\&-T_{1;q+1}\Big(x_3;(xx_2x_3)^{-1}\Big)+\ti_{q+2}\Big((xx_2)^{-1}\Big)\Big)
\\&+x\ti_1(x_3)\Big(T_{1,1;q}\Big(x_1,x_2;(xx_1x_2x_3)^{-1}\Big)-T_{1;q+1}\Big(x_1;(xx_1x_3)^{-1}\Big)\\&-T_{1;q+1}\Big(x_2;(xx_2x_3)^{-1}\Big)+\ti_{q+2}\Big((xx_3)^{-1}\Big)\Big)
\\&+xT_{1,1;q+1}\Big(x_1,x_2;(xx_1x_2)^{-1}\Big)+xT_{1,1;q+1}\Big(x_1,x_3;(xx_1x_3)^{-1}\Big)
\\&+xT_{1,1;q+1}\Big(x_2,x_3;(xx_2x_3)^{-1}\Big)-xT_{1;q+2}\Big(x_1;(xx_1)^{-1}\Big)\\&-xT_{1;q+2}\Big(x_2;(xx_2)^{-1}\Big)-xT_{1;q+2}\Big(x_3;(xx_3)^{-1}\Big)
\\&+(-1)^q(x_1x_2x_3)^{-1}\ti_1(x_1)\ti_1(x_2)\ti_1(x_3)\ti_q(xx_1x_2x_3)\\&-(-1)^q(x_1x_2)^{-1}\ti_1(x_1)\ti_1(x_2)T_{1;q}\Big(x_3^{-1};xx_1x_2x_3\Big)
\\&-(-1)^q(x_1x_3)^{-1}\ti_1(x_1)\ti_1(x_3)T_{1;q}\Big(x_2^{-1};xx_1x_2x_3\Big)\\&-(-1)^q(x_2x_3)^{-1}\ti_1(x_2)\ti_1(x_3)T_{1;q}\Big(x_1^{-1};xx_1x_2x_3\Big)
\\&+(-1)^qx_1^{-1}\ti_1(x_1)T_{1,1;q}\Big(x_2^{-1},x_3^{-1};xx_1x_2x_3\Big)\\&+(-1)^qx_2^{-1}\ti_1(x_2)T_{1,1;q}\Big(x_1^{-1},x_3^{-1};xx_1x_2x_3\Big)
\\&+(-1)^qx_3^{-1}\ti_1(x_3)T_{1,1;q}\Big(x_1^{-1},x_2^{-1};xx_1x_2x_3\Big)+(-1)^{q}\ti_{q+3}(x)\\&+(-1)^q\Li_{q}(xx_1x_2x_3)\Big(\ti_{3}(x)-x\ti_{3}\Big(x^{-1}\Big)\Big)
\\&+\frac{(-1)^qq(q+1)}{2}\Li_{q+2}(xx_1x_2x_3)\Big(\ti_{1}(x)-x\ti_{1}\Big(x^{-1}\Big)\Big)\\&-(-1)^qq\Li_{q+1}(xx_1x_2x_3)\Big(\ti_{2}(x)+x\ti_{2}\Big(x^{-1}\Big)\Big)
\\&+(-1)^qq\Big(\ti_{1}(x)-x\ti_{1}\Big(x^{-1}\Big)\Big)\Big\{\Li_1(x_1)\Li_{q+1}(xx_1x_2x_3)+\Li_1(x_2)\Li_{q+1}(xx_1x_2x_3)
\\&\qquad\qquad\qquad\qquad\qquad\qquad+\Li_1(x_3)\Li_{q+1}(xx_1x_2x_3)-S_{1;q+1}\Big(x_1^{-1};xx_1x_2x_3\Big)\\&\qquad\qquad\qquad\qquad\qquad\qquad-S_{1;q+1}\Big(x_2^{-1};xx_1x_2x_3\Big)-S_{1;q+1}\Big(x_3^{-1};xx_1x_2x_3\Big)\Big\}
\\&-(-1)^q\Big(\ti_{1}(x)-x\ti_{1}\Big(x^{-1}\Big)\Big)\Big\{\Li_2(x_1)\Li_{q}(xx_1x_2x_3)+\Li_2(x_2)\Li_{q}(xx_1x_2x_3)
\\&\qquad\qquad\qquad\qquad\qquad\qquad+\Li_2(x_3)\Li_{q}(xx_1x_2x_3)+S_{2;q}\Big(x_1^{-1};xx_1x_2x_3\Big)\\&\qquad\qquad\qquad\qquad\qquad\qquad+S_{2;q}\Big(x_2^{-1};xx_1x_2x_3\Big)+S_{2;q}\Big(x_3^{-1};xx_1x_2x_3\Big)\Big\}
\\&-(-1)^q\Big(\ti_{2}(x)+x\ti_{2}\Big(x^{-1}\Big)\Big)\Big\{\Li_1(x_1)\Li_{q}(xx_1x_2x_3)+\Li_1(x_2)\Li_{q}(xx_1x_2x_3)
\\&\qquad\qquad\qquad\qquad\qquad\qquad+\Li_1(x_3)\Li_{q}(xx_1x_2x_3)-S_{1;q}\Big(x_1^{-1};xx_1x_2x_3\Big)\\&\qquad\qquad\qquad\qquad\qquad\qquad-S_{1;q}\Big(x_2^{-1};xx_1x_2x_3\Big)-S_{1;q}\Big(x_3^{-1};xx_1x_2x_3\Big)\Big\}
\\&+(-1)^q\Big(\ti_{1}(x)-x\ti_{1}\Big(x^{-1}\Big)\Big)\Big\{\Li_1(x_1)\Li_1(x_2)\Li_q(xx_1x_2x_3)+\Li_1(x_1)\Li_1(x_3)\Li_q(xx_1x_2x_3)
\\&\qquad\qquad\qquad+\Li_1(x_2)\Li_1(x_3)\Li_q(xx_1x_2x_3)-\Li_1(x_1)S_{1;q}\Big(x_2^{-1},xx_1x_2x_3\Big)
\\&\qquad\qquad\qquad-\Li_1(x_1)S_{1;q}\Big(x_3^{-1},xx_1x_2x_3\Big)-\Li_1(x_2)S_{1;q}\Big(x_2^{-1},xx_1x_2x_3\Big)
\\&\qquad\qquad\qquad-\Li_1(x_2)S_{1;q}\Big(x_3^{-1},xx_1x_2x_3\Big)-\Li_1(x_3)S_{1;q}\Big(x_1^{-1},xx_1x_2x_3\Big)\\&\qquad\qquad\qquad-\Li_1(x_3)S_{1;q}\Big(x_2^{-1},xx_1x_2x_3\Big)+S_{1,1;q}\Big(x_1^{-1},x_2^{-1};xx_1x_2x_3\Big)
\\&\qquad\qquad\qquad+S_{1,1;q}\Big(x_1^{-1},x_3^{-1};xx_1x_2x_3\Big)++S_{1,1;q}\Big(x_2^{-1},x_3^{-1};xx_1x_2x_3\Big)\Big\}
\\&+\sum\limits_{m+k+1=q-2\atop m,k\geq0}(-1)^k\Big(\Li_{k+1}(x_1)+\Li_{k+1}(x_2)+\Li_{k+1}(x_3)\Big)\Big((-1)^m\ti_{m+1}(x)-x\ti_{m+1}\Big(x^{-1}\Big)\Big)
\\&+\sum\limits_{m+k_1+k_2+k_3=q-1\atop m,k_1,k_2,k_3\geq0}(-1)^{k_1+k_2+k_3}\Li_{k_1+1}(x_1)\Li_{k_2+1}(x_2)\Li_{k_3+1}(x_3)\Big((-1)^m\ti_{m+1}(x)-x\ti_{m+1}\Big(x^{-1}\Big)\Big)
\\&+\sum\limits_{m+k_1+k_2=q\atop m,k_1,k_2\geq0}(-1)^{k_1+k_2}\Big(\Li_{k_1+1}(x_1)\Li_{k_2+1}(x_2)+\Li_{k_1+1}(x_1)\Li_{k_2+1}(x_3)+\Li_{k_1+1}(x_3)\Li_{k_2+1}(x_2)\Big)
\\&\qquad\qquad\qquad\qquad\qquad\qquad\qquad\qquad\qquad\qquad\qquad\times\Big((-1)^m\ti_{m+1}(x)-x\ti_{m+1}\Big(x^{-1}\Big)\Big).
\end{align*}
\end{exa}
\begin{re} Theorem \ref{thm-parityc-theorem} is obtained by replacing $x$ with $(xx_1\cdots x_r)^{-1}$ in Theorem \ref{thm-parityc-CES-one}.
Indeed, the method of contour integration is entirely capable of yielding an explicit formula for Theorem \ref{thm-parityc-CES-one}, although given the complexity of the resulting expression, we have refrained from computing it in explicit detail. In \cite{CharltonHoffman-MathZ2025}, Charlton and Hoffman established the symmetry theorem for regularized multiple $t$-values and more general results, while the parity results for multiple $t$-values can also be derived from his paper by utilizing stuffle relations (\cite{H2000}).
\end{re}

\section{Further Remark}
In fact, we can consider contour integrals of the form
\begin{align}\label{contourintefinallyallrat}
\lim_{R\rightarrow \infty}\oint_{C_R} \frac{\Phi(s;x)\phi^{(p_1-1)}(s+1/2;x_1)\cdots\phi^{(p_r-1)}(s+1/2;x_r)}{(p_1-1)!\cdots (p_r-1)!}r_j(s)(-1)^{p_1+\cdots+p_r-r} ds=0\quad (j=1,2),
\end{align}
to study the parity of other types of cyclotomic Euler $T$-sums. Here $r_1(s)$ is a rational function in $s$ that has no poles at $-(n+1/2)\ (n\in \N)$ and $n\ (n\in \Z)$, while $r_2(s)$ is a rational function in $s$ that has no poles at $-(n+1/2)\ (n\in \N_0)$ and $n\ (n\in \Z\setminus \{0\})$, and $r_1(s),r_2(s)$ are $o(1)$ at infinity. Denote by $S_1$ and $S_2$ the set of poles of $r_1(s)$ and $r_2(s)$ respectively. For example, by examining the linear cases
\begin{align*}
\lim_{R\rightarrow \infty}\oint_{C_R} \frac{\Phi(s;x)\phi^{(p-1)}(s+1/2;y)}{(p-1)!}r_j(s)(-1)^{p-1} ds=0
\end{align*}
we can derive the following two general formulas:
\begin{thm}\label{thmsecfinall-one}
Let $x,y$ be roots of unity and $p_1\geq 1$. Assuming that $r_1(s)$ has a pole at $-1/2$ of order $q_1\geq0$ and $r_2(s)$ has a pole at $0$ of order $q_2\geq0$, then we have
\begin{align}\label{eq:linear-r1}
&-\sum\limits_{\alpha\in S_1\setminus\{-1/2\}}\Res\left(\frac{\Phi(s;x)\phi^{(p_1-1)}(s+1/2;y)}{(p_1-1)!}r_1(s)(-1)^{p_1-1} ,\alpha\right)\notag\\
&=\sum\limits_{n=0}^\infty x^{-n}y^{-n-1}r_1(n)\left(\ti_{p_1}(y)-t_n(p_1;y)\right)+\sum\limits_{n=1}^\infty (xy)^nr_1(-n)\left(\ti_{p_1}(y)y^{-1}+(-1)^{p_1} t_n\Big(p_1;y^{-1}\Big)\right)\notag\\
&\quad+\sum\limits_{m=0}^{p_1+q_1-1}\left((-1)^m \ti_{m+1}(x)-x\ti_{m+1}\Big(x^{-1}\Big) \right)\frac{R_1^{(p_1+q_1-m-1)}\left(-\frac{1}{2}\right)}{(p_1+q_1-1-m)!}\notag\\
&\quad+\sum_{m+k\leq q_1-1\atop q_1\geq1,m,k\geq 0} (-1)^k\binom{k+p_1-1}{p_1-1}\Li_{k+p_1}(y)\left((-1)^m \ti_{m+1}(x)-x\ti_{m+1}\Big(x^{-1}\Big) \right)\frac{R_1^{(q_1-m-k-1)}\left(-\frac{1}{2}\right)}{(q_1-m-k-1)!}\notag\\
&\quad+\sum\limits_{n=1}^\infty \sum_{m=0}^{p_1-1} \left((-1)^m \ti_{m+1}(x)-x \ti_{m+1}\Big(x^{-1}\Big) \right) \frac{(xy)^nr_1^{(p_1-m-1)}\left(-n-\frac{1}{2}\right)}{(p_1-m-1)!}
\end{align}
and
\begin{align}\label{eq:linear-r2}
&-\sum\limits_{\alpha\in S_2\setminus\{0\}}\Res\left(\frac{\Phi(s;x)\phi^{(p_1-1)}(s+1/2;y)}{(p_1-1)!}r_2(s)(-1)^{p_1-1} ,\alpha\right)\notag\\
&=\sum\limits_{n=1}^\infty x^{-n}y^{-n-1}r_2(n)\left(\ti_{p_1}(y)-t_n(p_1;y)\right)+\sum\limits_{n=1}^\infty (xy)^nr_2(-n)\left(\ti_{p_1}(y)y^{-1}+(-1)^{p_1} t_n\Big(p_1;y^{-1}\Big)\right)\notag\\
&\quad+\sum\limits_{k=0}^{q_2}\binom{k+p_1-1}{p_1-1}(-1)^k\ti_{k+p_1}(y)y^{-1}\frac{R_2^{(q_2-k)}(0)}{(q_2-k)!}\notag\\
&\quad+\sum_{m+k\leq q_2-1\atop q_2\geq1,m,k\geq 0} (-1)^k\binom{k+p_1-1}{p_1-1}\ti_{k+p_1}(y)y^{-1}\left((-1)^m \Li_{m+1}(x)-\Li_{m+1}\Big(x^{-1}\Big) \right)\frac{R_2^{(q_2-m-k-1)}(0)}{(q_2-m-k-1)!}\notag\\
&\quad+\sum\limits_{n=0}^\infty \sum_{m=0}^{p_1-1} \left((-1)^m \ti_{m+1}(x)-x \ti_{m+1}\Big(x^{-1}\Big) \right) \frac{(xy)^nr_2^{(p_1-m-1)}\left(-n-\frac{1}{2}\right)}{(p_1-m-1)!}
\end{align}
where $R_1(s)=(s+1/2)^{q_1}r_1(s)$ and $R_2(s)=s^{q_2}r_2(s)$.
\end{thm}
\begin{proof}
First, considering the residue calculations of the following contour integral:
\begin{align*}
\lim_{R\rightarrow \infty}\oint_{C_R} F_{p_1,q_1}(x,y;s)ds:= \lim_{R\rightarrow \infty}\oint_{C_R} \frac{\Phi(s;x)\phi^{(p_1-1)}(s+1/2;y)}{(p_1-1)!}r_1(s) (-1)^{p_1-1}ds=0.
\end{align*}
The integrand $F_{p_1,q_1}(x,y;s)$ has the following poles throughout the complex plane: 1. All integers (simple poles); 2. $-1/2$ (pole of order $p_1+q_1)$ and 3. $-(n+1/2)$ (for positive integer $n$, poles of order $p_1$). Applying Lemmas \ref{lem-rui-xu-one}-\ref{lem-extend-rui-xu-two}, by direct calculations, we deduce the following residues
\begin{align*}
&\Res\left(F_{p_1,q_1}(x,y;s),n\right)=x^{-n}y^{-n-1}r_1(n)\left(\ti_{p_1}(y)-t_n(p_1;y)\right)\quad (n\geq 0),\\
&\Res\left(F_{p_1,q_1}(x,y;s),-n\right)=(xy)^nr_1(-n)\left(\ti_{p_1}(y)y^{-1}+(-1)^{p_1} t_n\Big(p_1;y^{-1}\Big)\right)\quad (n\geq 1),\\
&\Res\left(F_{p_1,q_1}(x,y;s),-n-1/2\right)=\frac1{(p_1-1)!} \lim_{s\rightarrow -n-1/2} \frac{d^{p_1-1}}{ds^{p_1-1}}\left((s+n+1/2)^{p_1}F_{p_1,q_1}(x,y;s)\right)\\
&=(xy)^n\sum_{m=0}^{p_1-1} \left((-1)^m \ti_{m+1}(x)-x \ti_{m+1}\Big(x^{-1}\Big) \right) \frac{r_1^{(p_1-m-1)}\left(-n-\frac{1}{2}\right)}{(p_1-m-1)!}\quad (n\geq 1)
\end{align*}
and
\begin{align*}
&\Res\left(F_{p_1,q_1}(x,y;s),-1/2\right)=\frac1{(p_1+q_1-1)!} \lim_{s\rightarrow -1/2} \frac{d^{p_1+q_1-1}}{ds^{p_1+q_1-1}}\left((s+1/2)^{p_1+q_1}F_{p_1,q_1}(x,y;s)\right)\\
&=\sum\limits_{m=0}^{p_1+q_1-1}\left((-1)^m \ti_{m+1}(x)-x\ti_{m+1}\Big(x^{-1}\Big) \right)\frac{R_1^{(p_1+q_1-m-1)}\left(-\frac{1}{2}\right)}{(p_1+q_1-1-m)!}\\
&\quad+\sum_{m+k\leq q_1-1\atop q_1\geq1,m,k\geq 0} (-1)^k\binom{k+p_1-1}{p_1-1}\Li_{k+p_1}(y)\left((-1)^m \ti_{m+1}(x)-x\ti_{m+1}\Big(x^{-1}\Big) \right)\frac{R_1^{(q_1-m-k-1)}\left(-\frac{1}{2}\right)}{(q_1-m-k-1)!}.
\end{align*}

From Lemma \ref{lem-redisue-thm}, we know that
\begin{align*}
&\Res\left(F_{p_1,q_1}(x,y;s),-1/2\right)+\sum_{n=1}^\infty \Res\left(F_{p_1,q_1}(x,y;s),-n-1/2\right) +\sum_{n=0}^\infty \Res\left(F_{p_1,q_1}(x,y;s),n\right) \\
&\quad+\sum_{n=1}^\infty \Res\left(F_{p_1,q_1}(x,y;s),-n\right)+\sum\limits_{\alpha\in S_1\setminus\{-1/2\}}\Res(F_{p_1,q_1}(x,y;s),\alpha)=0.
\end{align*}
Finally, combining these contributions yields the result \eqref{eq:linear-r1}.

Considering the residue calculations of the following contour integral:
\begin{align*}
\lim_{R\rightarrow \infty}\oint_{C_R} F_{p_1,q_2}(x,y;s)ds:= \lim_{R\rightarrow \infty}\oint_{C_R} \frac{\Phi(s;x)\phi^{(p_1-1)}(s+1/2;y)}{(p_1-1)!}r_2(s) (-1)^{p_1-1}ds=0.
\end{align*}
The integrand $F_{p_1,q_2}(x,y;s)$ has the following poles throughout the complex plane: 1. All nonzero integers (simple poles); 2. $0$ (pole of order $q_2+1$)and 3. $-(n+1/2)$ (for nonnegative integer $n$, poles of order $p_1$). Applying Lemmas \ref{lem-rui-xu-one}-\ref{lem-extend-rui-xu-two}, by direct calculations, we deduce the following residues
\begin{align*}
&\Res\left(F_{p_1,q_2}(x,y;s),n\right)=x^{-n}y^{-n-1}r_2(n)\left(\ti_{p_1}(y)-t_n(p_1;y)\right)\quad (n\geq 1),\\
&\Res\left(F_{p_1,q_2}(x,y;s),-n\right)=(xy)^nr_2(-n)\left(\ti_{p_1}(y)y^{-1}+(-1)^{p_1} t_n\Big(p_1;y^{-1}\Big)\right)\quad (n\geq 1),\\
&\Res\left(F_{p_1,q_2}(x,y;s),-n-1/2\right)=\frac1{(p_1-1)!} \lim_{s\rightarrow -n-1/2} \frac{d^{p_1-1}}{ds^{p_1-1}}\left((s+n+1/2)^{p_1}F_{p_1,q_2}(x,y;s)\right)\\
&=(xy)^n\sum_{m=0}^{p_1-1} \left((-1)^m \ti_{m+1}(x)-x \ti_{m+1}\Big(x^{-1}\Big) \right) \frac{r_2^{(p_1-m-1)}\left(-n-\frac{1}{2}\right)}{(p_1-m-1)!}\quad (n\geq 0)
\end{align*}
and
\begin{align*}
&\Res\left(F_{p_1,q_2}(x,y;s),0\right)=\frac1{q_2!} \lim_{s\rightarrow 0} \frac{d^{q_2}}{ds^{q_2}}\left(s^{q_2+1}F_{p_1,q_2}(x,y;s)\right)\\
&=\sum\limits_{k=0}^{q_2}\binom{k+p_1-1}{p_1-1}(-1)^k\ti_{k+p_1}(y)y^{-1}\frac{R_2^{(q_2-k)}(0)}{(q_2-k)!}\\
&\quad+\sum_{m+k\leq q_2-1\atop q_2\geq1,m,k\geq 0} (-1)^k\binom{k+p_1-1}{p_1-1}\ti_{k+p_1}(y)y^{-1}\left((-1)^m \Li_{m+1}(x)-\Li_{m+1}\Big(x^{-1}\Big) \right)\frac{R_2^{(q_2-m-k-1)}(0)}{(q_2-m-k-1)!}.
\end{align*}
Similarly, combining these contributions and Lemma \ref{lem-redisue-thm} yields the statement of \eqref{eq:linear-r2}.
\end{proof}

Obviously, by setting $r_1(s) = (s + 1/2)^{-q}$ in equation \eqref{eq:linear-r1} of Theorem \ref{thmsecfinall-one}, a direct residue computation yields Theorem \ref{thm-linearETS}.

The quadratic cases can be derived by evaluating the contour integral
\begin{align}\label{quadraticCMTV-ES}
\lim_{R\rightarrow \infty}\oint_{C_R} \frac{\Phi(s;x)\phi^{(p_1-1)}(s+1/2;x_1)\phi^{(p_2-1)}(s+1/2;x_2)}{(p_1-1)! (p_2-1)!}r_j(s)(-1)^{p_1+p_2} ds=0.
\end{align}
We leave the details of this calculation to interested readers.

When $r_2(s):=1/s^q\ (q\in\N)$, the contour integral \eqref{contourintefinallyallrat} can be utilized to study the parity of the cyclotomic version of equation \eqref{Stsum.Unify}, which is related to the cyclotomic version of Kaneko-Tsumura's multiple $T$-values. We define the cyclotomic version of \eqref{Stsum.Unify} $\tilde{S}_{p_1,\ldots, p_r;q}(x_1,\ldots,x_r;x)$ and the cyclotomic version of Kaneko-Tsumura's multiple $T$-values $\Ti_{k_1,\ldots,k_r}(x_1,\ldots,x_r)$ as follows:
\begin{align}
&\tilde{S}_{p_1,\ldots, p_k;q}(x_1,\ldots,x_r;x):=\sum_{n=1}^\infty \frac{t_{n}(p_1;x_1)t_{n}(p_2;x_2)\cdots t_{n}(p_r;x_r)}{n^q}x^n,\label{defnS-ES2}
\end{align}
and
\begin{align}\label{defn-CMTV}
\Ti_{k_1,\ldots,k_r}(x_1,\ldots,x_r):=2^r\sum_{0<n_1<n_2<\cdots<n_r} \frac{x_1^{n_1}x_2^{n_2}\cdots x_r^{n_r}}{(2n_1-1)^{k_1}(2n_2-2)^{k_2}\cdots (2n_r-r)^{k_r}}
\end{align}
where $p_1,\ldots,p_r,q\in \N,\ k_1,\ldots,k_r\in \N$ and $x_1,\ldots,x_r,x$ are all roots of unity with $(q,x)\neq (1,1)$ and $(k_r,x_r)\neq (1,1)$. In particular, if $x_1,\ldots,x_r$ in \eqref{defn-CMTV} are all $N$-th roots of unity, they are referred to as \emph{level $N$ cyclotomic multiple $T$-values}. Clearly, the multiple series on the right hand side of \eqref{defn-CMTV} also converges for $|x_j\cdots x_r|<1\ (j=1,2,\ldots,r)$, in which case we call the series a \emph{multiple $T$-polylogarithm function}. From the definitions of both, the following relationships can be readily obtained:
\begin{align*}
&\Ti_{p,q}(x,y)=\frac{y}{2^{p+q-2}}\tilde{S}_{p;q}(x;y),\\
&\Ti_{p,q,r}(x,y,z)=\frac{yz}{2^{p+q+r-3}}\ti_r(z)\tilde{S}_{p;q}(x;y)-\frac{yz}{2^{p+q+r-3}}\tilde{S}_{p,r;q}(x,z;y).
\end{align*}
When $r_2(s) = s^{-q}$ in \eqref{eq:linear-r2}, the resulting outcome can then be used to investigate the parity results of cyclotomic double $T$-values. Similarly, by considering $r_2(s) = s^{-q}$ in equation \eqref{quadraticCMTV-ES}, the results obtained through the computation of its residue values can be used to study the parity of cyclotomic triple $T$-values. We leave the detailed process to interested readers. Thus, we are able to provide the following statements regarding the parity of cyclotomic multiple $T$-values at depths two and three:
\begin{cor} Let $x,y$ be roots of unity, and $p,q\geq 1$ with $(p,x), (q,y)\neq (1,1)$. Then
\begin{align*}
\Ti_{p,q}(x,y)-(-1)^{p+q}xy^2\Ti_{p,q}\Big(x^{-1},y^{-1}\Big)
\end{align*}
can be expressed in terms of a rational combination of products of cyclotomic single $T$-values and cyclotomic single zeta values.
\end{cor}

\begin{cor} Let $x,y,z$ be roots of unity, and $p,q\geq 1$ with $(p,x), (q,y), (r,z)\neq (1,1)$. Then
\begin{align*}
\Ti_{p,q,r}(x,y,z)+(-1)^{p+q+r}xy^2z^3\Ti_{p,q,r}\Big(x^{-1},y^{-1},z^{-1}\Big)
\end{align*}
can be expressed in terms of a rational combination of products of cyclotomic multiple $T$-values and cyclotomic multiple zeta values with depth $\leq 2$.
\end{cor}
From the definitions of cyclotomic multiple $T$-values and cyclotomic multiple zeta values, it is straightforward to observe that
\begin{align*}
\Ti_{k_1,\ldots,k_r}(x_1,\ldots,x_r)= (\sqrt{x_1})(\sqrt{x_2})^2\cdots (\sqrt{x_r})^r\sum_{\si_1,\ldots,\si_r\in\{\pm 1\}}\si_1 \si_2^2\cdots \si_r^r \Li_{k_1,\ldots,k_r}(\si_1\sqrt{x_1},\ldots,\si_r\sqrt{x_r}).
\end{align*}
Furthermore, by applying Panzer's parity result for multiple polylogarithms, the following parity conclusion regarding cyclotomic multiple $T$-values can be established:
\begin{thm}\label{thm-weakendedofCMTVparity}
Let $r>1$ and $x_1,\ldots,x_r$ be roots of unity, and $k_1,\ldots,k_r\geq 1$ with $(k_r,x_r)\neq (1,1)$. If $x_1,\ldots,x_r\in\{z\in \mathbb{C}: z^N=1\}$, then
\begin{align*}
&\Ti_{k_1,\ldots,k_r}(x_1,\ldots,x_r)-(-1)^{k_1+\cdots+k_r+r}(x_1x_2^2\cdots x_r^r)\Ti_{k_1,\ldots,k_r}\Big(x_1^{-1},\ldots,x_r^{-1}\Big)
\end{align*}
can be expressed in terms of a $\Q$-linear combination of cyclotomic multiple zeta values with depth less than $r$ and level less than or equal to $2N$.
\end{thm}
\begin{proof} By a direct calculation, one obtains
\begin{align*}
&\Ti_{k_1,\ldots,k_r}(x_1,\ldots,x_r)-(-1)^{k_1+\cdots+k_r+r}(x_1x_2^2\cdots x_r^r)\Ti_{k_1,\ldots,k_r}\Big(x_1^{-1},\ldots,x_r^{-1}\Big)\\
&=(\sqrt{x_1})(\sqrt{x_2})^2\cdots (\sqrt{x_r})^r\sum_{\si_1,\ldots,\si_r\in\{\pm 1\}}\si_1\si_2^2 \cdots \si_r^r \\&\qquad\qquad\quad\times \left(\Li_{k_1,\ldots,k_r}(\si_1\sqrt{x_1},\ldots,\si_r\sqrt{x_r})-(-1)^{k_1+\cdots+k_r+r}\Li_{k_1,\ldots,k_r}\left(\frac1{\si_1\sqrt{x_1}},\ldots,\frac1{\si_r\sqrt{x_r}}\right)\right).
\end{align*}
Therefore, by applying Panzer's parity theorem, the proof of the theorem can be completed.
\end{proof}
Finally, we propose one conjecture and one question regarding the parity of cyclotomic multiple $T$-values:
\begin{con}
Let $r>1$ and $x_1,\ldots,x_r$ be roots of unity, and $k_1,\ldots,k_r\geq 1$ with $(k_r,x_r)\neq (1,1)$. If $x_1,\ldots,x_r\in\{z\in \mathbb{C}: z^N=1\}$, then
\begin{align*}
&\Ti_{k_1,\ldots,k_r}(x_1,\ldots,x_r)-(-1)^{k_1+\cdots+k_r+r}(x_1x_2^2\cdots x_r^r)\Ti_{k_1,\ldots,k_r}\Big(x_1^{-1},\ldots,x_r^{-1}\Big)
\end{align*}
can be expressed in terms of a rational combination of products of cyclotomic multiple $T$-values and cyclotomic multiple zeta values with depth $\leq r-1$ and level $\leq N$.
\end{con}

\begin{qu}
Similar to multiple polylogarithms, can the multiple $T$-polylogarithm function $\Ti_{\bfk}(\bfx)$ be analytically continued to the complex plane, yielding a generalization analogous to Panzer's parity theorem for multiple polylogarithms applied to the analytically continued multiple $T$-polylogarithm function?
\end{qu}

\medskip

{\bf Declaration of competing interest.}
The authors declares that they has no known competing financial interests or personal relationships that could have
appeared to influence the work reported in this paper.

{\bf Data availability.}
No data was used for the research described in the article.

{\bf Acknowledgments.}  Both authors of this paper sincerely thank Professor Masanobu Kaneko, their advisor during their studies at Kyushu University in Japan, for his consistent support and guidance. They wish Professor Kaneko a happy 65th birthday. Ce Xu is supported by the General Program of Natural Science Foundation of Anhui Province (Grant No. 2508085MA014). Ce Xu gratefully acknowledges the invitation from Professor Chengming Bai of Nankai University to the Chern Institute of Mathematics and from Professor Shaoyun Yi of Xiamen University to the Tianyuan Mathematical Center in Southeast China (TMSE). This work commenced during these visits.

\end{document}